\documentclass[reqno,11pt]{amsart}
\usepackage{amsmath,amssymb,amsfonts}

\usepackage{tikz}
\usetikzlibrary{matrix,arrows}
\usepackage{verbatim}

\usepackage[pagebackref,colorlinks,citecolor={red!50!black},linkcolor={blue!80!black}]{hyperref}

\renewcommand*{\backref}[1]{}
\renewcommand*{\backrefalt}[4]{%
    \scriptsize%
    {
    \ifcase #1 (\textcolor{red}{Uncited.})%
          \or (Cited\ on p.~#2)%
          \else (Cited\ on pp.~#2)%
    \fi%
    }
}

\usepackage[mathcalasscript,partialup]{kpfonts}
\usepackage[utf8]{inputenc}

\title{Unbounded algebraic derivators}

\author[L. Alonso]{Leovigildo Alonso Tarr\'{\i}o}
\address[L. A. T.]{CITMAGA\\
Departamento de Matem\'a\-ticas\\
Universidade de Santiago de Compostela\\
E-15782  Santiago de Compostela, Spain}
\email{leo.alonso@usc.es}

\author[B. \'Alvarez]{Beatriz \'Alvarez D\'iaz}
\address[B. A. D.]{CITMAGA\\
Departamento de Matem\'a\-ticas\\
Universidade de Santiago de Compostela\\
E-15782  Santiago de Compostela, Spain}
\email{beatrizalvarez.diaz@usc.es}

\author[A. Jerem\'{\i}as]{Ana Jerem\'{\i}as L\'opez}
\address[A. J. L.]{CITMAGA\\
Departamento de Matem\'a\-ticas\\
Universidade de Santiago de Compostela\\
E-15782  Santiago de Compostela, Spain}
\email{ana.jeremias@usc.es}

\subjclass[2010]{18G80 (primary); 18G35, 13D45, 20J06 (secondary)}

\date{\emph{Released}: September 13, 2023; \emph{typeset}: \today}

\theoremstyle{plain}
\newtheorem{thm}{Theorem}[section]
\newtheorem{lem}[thm]{Lemma}
\newtheorem{cor}[thm]{Corollary}
\newtheorem{prop}[thm]{Proposition}

\theoremstyle{remark}
\newtheorem*{rem}{Remark}

\theoremstyle{definition}

\newtheorem*{ack}{Acknowledgements}
\newtheorem{cosa}[thm]{}

\newtheorem*{fund}{Funding}

\numberwithin{equation}{thm}

\newcommand{\SA}{\mathsf{A}}
\newcommand{\SB}{\mathsf{B}}
\newcommand{\SC}{\mathsf{C}}

\newcommand{\SH}{\mathsf{H}}
\newcommand{\SI}{\mathsf{I}}
\newcommand{\SJ}{\mathsf{J}}
\newcommand{\SK}{\mathsf{K}}
\newcommand{\SL}{\mathsf{L}}

\newcommand{\ST}{\boldsymbol{\mathsf{T}}}

\newcommand{\CCC}{\boldsymbol{\mathsf{C}}}
\renewcommand{\D}{\boldsymbol{\mathsf{D}}}

\newcommand{\LL}{\boldsymbol{\mathsf{L}}}
\newcommand{\R}{\boldsymbol{\mathsf{R}}}
\newcommand{\A}{\mathsf{A}}

\newcommand{\es}{\boldsymbol{\mathsf{1}}}
\newcommand{\ts}{\mathsf{t}}

\newcommand{\NN}{\mathbb{N}}
\newcommand{\ZZ}{\mathbb{Z}}

\newcommand{\ia}{{\mathfrak a}}

\newcommand{\ip}{{\mathfrak p}}

\newcommand{\ggmm}{\boldsymbol{\gamma}}
\newcommand{\aall}{\boldsymbol{\alpha}}
\newcommand{\SGamma}{\mathsf{\Gamma}\!}

\newcommand{\dirlim}[1]{\begin{array}[t]{c} {\rm lim}\\[-7.5 pt]
 {\longrightarrow} \\[-7.5 pt] {\scriptstyle {#1}} \end{array}}
\newcommand{\invlim}[1]{\begin{array}[t]{c} {\rm lim}\\[-7.5 pt]
 {\longleftarrow} \\[-7.5 pt] {\scriptstyle {#1}} \end{array}}
 
 \newcommand{\shdirlim}{\underset{\longrightarrow}{\mathrm{lim}}}
 \newcommand{\shinvlim}{\underset{\longleftarrow}{\mathrm{lim}}}

\newcommand{\lto}{\longrightarrow}
\newcommand{\xto}{\xrightarrow}

\newcommand{\lot}{\longleftarrow}
\newcommand{\xot}{\xleftarrow}

\newcommand{\epi}{\twoheadrightarrow}

\newcommand{\inc}{\hookrightarrow}
\newcommand{\iso}{\mathrel{\tilde{\to}}}

\newcommand{\liso}{\mathrel{\tilde{\lto}}}
\newcommand{\losi}{\mathrel{\tilde{\lot}}}

\newcommand*{\longrightrightarrows}{
\ensuremath{%
\makebox[0pt][l]
{\raisebox{0.3ex}{\ensuremath{\longrightarrow}}}%
\raisebox{-0.3ex}{\ensuremath{\longrightarrow}}%
}}

\newcommand{\toto}{\longrightrightarrows}

\newcommand{\imp}{\Rightarrow}

\DeclareMathOperator{\Hom}{Hom}

\DeclareMathOperator{\Img}{Im}

\DeclareMathOperator{\aut}{Aut}
\DeclareMathOperator{\spec}{Spec}

\DeclareMathOperator{\h}{H}
\DeclareMathOperator{\id}{id}
\DeclareMathOperator{\iid}{\mathsf{id}}

\DeclareMathOperator{\ran}{Ran}
\DeclareMathOperator{\lan}{Lan}
\DeclareMathOperator{\cyl}{Cyl}
\newcommand{\md}{\text{-}\mathsf{Mod}}

\newcommand{\op}{\mathsf{op}}

\newcommand{\dia}{\mathsf{dia}}
\newcommand{\ddia}{\boldsymbol{\mathsf{dia}}}
\newcommand{\loc}{\mathsf{S}}
\newcommand{\dos}{\boldsymbol{\mathsf{2}}}

\DeclareMathOperator{\cat}{\mathsf{Cat}}
\DeclareMathOperator{\bigcat}{\mathsf{CAT}}

\DeclareMathOperator{\bff}{\mathbf{f}}
\DeclareMathOperator{\bgg}{\mathbf{g}}
\DeclareMathOperator{\bhh}{\mathbf{h}}

\newcommand{\PDC}{\boldsymbol{\mathcal{C}}}

\newcommand{\PDD}{\boldsymbol{\mathcal{D}}}

\newcommand{\ul}{\text{\smash{\raisebox{-1ex}{\scalebox{1.4}{$\ulcorner$}}}}}
\newcommand{\lr}{\text{\smash{{\scalebox{1.5}{$\lrcorner$}}}}}

\newcommand{\ie}{{\it i.e.~}}
\newcommand{\cfr}{{\it cf.~}}

\begin{document}

\begin{abstract} 
We show that the unbounded derived category of a Grothendieck category with enough projective objects is the base category of a derivator whose category of diagrams is the full 2-category of small categories. With this structure, we give a description of the localization functor associated to a specialization closed subset of the spectrum of a commutative noetherian ring. In addition, using the derivator of modules, we prove some basic theorems of group cohomology for complexes of representations over an arbitrary base ring.
\end{abstract}

\maketitle
\tableofcontents

\section*{Introduction}

In this review paper we prove that the derived category of a Grothendieck category with enough projective objects is the base category of a derivator. The word ``unbounded'' in the title is used in two senses: first, we consider unbounded complexes of objects of the initial abelian category; and second, the 2-category of diagrams on which homotopy co/limits are defined, namely, the full 2-category of small categories. So, there is no bound neither on complexes nor in the shape of the co/limit diagrams.

The point of this paper is to give a new proof of a more general theorem of Cisinski in the setting of derived categories, relying exclusively on classical methods of homological algebra and avoiding any appearance of model categories. Let us recall Cisinki's result: it shows that a complete and cocomplete model category defines a derivator by associating to every small category the localization of the diagram category by the levelwise weak equivalences \cite[Th\'eor\`eme 6.11]{IDC}. This theorem is extended in \cite{LCF} by looking at the connection between the homotopy of correspondences of small categories and the bicategory of derivators with cocontinuous morphisms. In a further work, \cite{CDer} Cisinski extends the result giving minimal conditions for a category to define a derivator (with domain the finite directed categories). A nice exposition of Cisinki's theorem may be found in \cite{bld}.

The theory of derivators was developed by Grothendieck and Heller (under the name ``homotopy theories'') in order to study homotopy phenomena ``without models''. Stable derivators may be regarded as a minimal enhancement of triangulated categories, a well-established framework to study stable homotopy. The success of triangulated categories is well-known as this concept permits the study of quotients and localizations, and it is rich enough to provide theorems like Brown representability. However, the definition has also some well-known drawbacks whose paradigm is the non functoriality of the cone.

There are several model structures in categories of unbounded complexes. A good reference for this is \cite{cisdeg}, from which it follows that Cisinki's result is applicable in this setting. It is however completely classical to establish the basic properties of derived categories by means of homological algebra without recourse to a specific choice of model structure  bypassing the full theory of homotopical algebra. The basic results on unbounded derived categories were developed by Spaltenstein \cite{S}, and expanded by B\"okstedt and Neeman \cite{BN}. In both references, only homological algebra methods (additive homotopies, cones, cylinders and resolutions) are used.

In this work we stick to these methods. This approach provides some benefits. One is that it lowers the bar to get into this theory. Another one is that the structure inherent to the derivator is made explicit in terms of simple homological constructions. In the spirit of a review paper, we give two applications of these ideas. The first one gives the description of local cohomology by Koszul complexes. It generalizes the well-known case of a closed subset of the spectrum of a commutative noetherian ring to the case of a subset stable for specialization. The second one derives some basic results in group cohomology from the basic structure of the derivator.

Let us describe then with greater detail the precise contents of this work.

The first section is devoted to recall known results casted in a language and notations convenient for our purposes, and setting the basic notations that will be used throughout the paper. The main observation is that the diagram category of complexes over a Grothendieck category with enough projectives has this property too. Therefore all such categories possess all co/limits and can be organized in a representable derivator as stated in \ref{derC}. We use it to give a construction of the derivator associated to the derived category by deriving the relevant functors. We give a detailed description of the Kan extension functors $u_*$ and $u_!$ that will be of use along the paper.

Section \ref{dos} gives the verification of the axioms of derivator, see \cite[Definition 1.5]{gr13}. The axioms (Der1) and (Der2) follow from the structure already present in complexes. Axiom (Der3) is proved by exhibiting that the derived functors of $u_*$ and $u_!$  satisfy it, as they are defined by Kan extensions. Axiom (Der4) uses an argument on preservation of $q$-projective (and $q$-injective) objects for the functors from the co/slice categories presented in Lemma \ref{qconserv}. In the next section we prove that the derivator is strong (Der5) \cite[Definition 1.8]{gr13}. For this, we make an explicit homotopy rectification of squares using basic homological tools, namely, cylinders and homotopies.

In section \ref{cuatro} we prove our main result, Theorem \ref{main}: for a Grothendieck category $\SA$ with enough projective objects, its associated derivator $\PDD_{\!\!\SA}$ is stable. A consequence of this construction is that the classical cone construction agrees with the natural cone provided by the derivator, showing its true intrinsic nature.

The following sections complete the paper by giving applications of this framework. Section \ref{cinco} gives a description of the localization functor associated to a specialization closed subset $Y$ of the spectrum of a commutative noetherian ring $A$, see Theorem \ref{mainloc}. In the case where $Y$ is closed, countable homotopy colimits \emph{a la} B\"okstedt and Neeman \cite{BN} suffice. When $Y$ is more general we need a more general diagram of Koszul complexes to extend the description. Our methods show that this approach fits naturally in the context of derivators.

Finally, in section \ref{seis} we show that the framework of the derivator provides a more general (arbitrary base ring, complexes as coefficients) and simpler approach to some basic theorems of group cohomology. The derived Frobenius reciprocity is a restatement of (Der3). Shapiro's Lemma follows again from a standard argument on preservation of $q$-projective (and $q$-injective) objects. Finally, we give a version of Lyndon-Hochschild-Serre, interpreted as a composition of derived functors that follows from 2-functoriality of $(-)_*$ and $(-)_!$, see  Theorem \ref{LHSclas}.

Let us say a few words about the prerequisites. Throughout the paper we will use the basic language from \cite{currante}, especially concerning Kan extensions. Our basic reference for derived categories will be \cite{yellow}. The construction of resolutions in taken from the ideas in \cite{BN}, made explicit in \cite{AJS}. The conventions regarding derivators are taken from \cite{gr13}.

\section{The representable derivator of complexes}

\begin{cosa} \textbf{Grothendieck categories with enough projectives.}
Let us fix the basic standing assumption. 
 Let $\SA$ be a Grothendieck category, \ie an abelian category with a generator and satisfying the AB5 condition: small filtered colimits are exact. Furthermore, we will always assume that $\SA$ possesses enough projective objects. An equivalent condition for $\SA$ to have enough projectives is that there is a set of \emph{projective} generators. These categories arise naturally as categories of modules over ``rings with several objects''. The simplest and most important case is the category $R\md$, where $R$ is a ring (which we assume unital, as usual). For readers not interested in maximal generality we recommend to assume tacitly that $\SA = R\md$ throughout the paper. 
\end{cosa}

\begin{cosa} \textbf{Categories of diagrams.}
 For any category $\SA$ and any small category $\SI$ denote as usual by $\SA^{\SI}$ the category of functors from $\SI$ to $\SA$.  We will refer to the category $\SA^{\SI}$ as the \emph{category of diagrams in $\SA$ with shape} $\SI$.
 \end{cosa}

\begin{cosa} \textbf{Diagrams of complexes.}
 We denote, as usual, the category of complexes over $\SA$ by $\CCC(\SA)$. Notice the canonical isomorphism of categories $\CCC(\SA)^{\SI} \cong \CCC(\SA^{\SI})$ that we will use  without further notice along the paper. 
 For a functor $u \colon \SI \to \SJ$, that represents a change of shape of diagrams, and a functor $F \colon \SJ \to \CCC(\SA)$ (a diagram) we associate $u^*F \colon \SI \to \CCC(\SA)$ defined by $u^*F := F \circ u$. This defines a functor
\[
u^* \colon \CCC(\SA)^{\SJ} \lto \CCC(\SA)^{\SI}
\]
from diagrams of shape ${\SJ}$ to diagrams of shape ${\SI}$. 
\end{cosa}

\begin{cosa} \textbf{The derivator of complexes.} \label{derC}
Associated to $\SA$, consider the (strict) 2-functor
\[
\PDC_{\!\!\SA} \colon \cat^\op \lto \bigcat
\]
defined by $\PDC_{\!\!\SA}(\SI) := \CCC(\SA^{\SI})$, $\PDC_{\!\!\SA }(u) := u^*$ and similarly for natural transformations. Using the notation given in \cite[Example 1.2 (1)]{gr13} this is the representable prederivator $\PDC_{\!\!\SA } = y(\CCC(\SA))$.

A Grothendieck category $\SA$ has all (small) limits and colimits, and so does its associated category of complexes $\CCC(\SA)$. This property gives $\PDC_{\!\!\SA }$ the structure of derivator. Next we spell out the relevant structure. Denote by $\es$ the terminal category. Let us describe first the action for the canonical functor $c \colon \SI \to \es$. Note that $c^*$ is the constant system indexed by $\SI$ and we have the adjunctions $\shdirlim \dashv c^* \!\dashv \shinvlim{}$ as in \cite[X.1]{currante} (where $c^*$ is denoted $\Delta$)
\[
\begin{tikzpicture}
      \node (G) {$\PDC_{\!\!\SA }(\SI)$};  
      \node[node distance=3cm, right of = G] (H) {$\PDC_{\!\!\SA }(\es)$};
      \draw[->, bend left=30] (G) to node [above] {\footnotesize $\shdirlim{}$} (H);
      \draw[->, bend right=30] (G) to node [below] {\footnotesize $\shinvlim{}$} (H);
      \draw[<-,] (G) to node [above] {\footnotesize $c^*$} (H);
\end{tikzpicture}
\]
In general, let $u \colon \SI \to \SJ$ be a functor and let
\[
u^* \colon \CCC(\SA)^{\SJ} \lto \CCC(\SA)^{\SI}
\]
be its associated natural transformation. This functor is interpreted as
\[
u^* \colon \PDC_{\!\!\SA }(\SJ) \lto \PDC_{\!\!\SA }(\SI)
\]
An alternative notation for $u^*$ might be $\PDC_{\!\!\SA }(u)$ but we will stick to the usual one.
\end{cosa}

\begin{cosa} \label{expdesadj} 
\textbf{Adjoints to the change of shape functor.}
As before, let $u \colon \SI \to \SJ$ be a functor between small categories. The functor $u^* \colon \PDC_{\!\!\SA }(\SJ) \to \PDC_{\!\!\SA }(\SI)$ possesses an adjoint at either side that may be depicted as:
\[
\begin{tikzpicture}
      \node (G) {$\PDC_{\!\!\SA }(\SI)$};  
      \node[node distance=3cm, right of = G] (H) {$\PDC_{\!\!\SA }(\SJ)$};
      \draw[->, bend left=30] (G) to node [above] {\footnotesize $u_!$} (H);
      \draw[->, bend right=30] (G) to node [below] {\footnotesize $u_*$} (H);
      \draw[<-,] (G) to node [above] {\footnotesize $u^*$} (H);
\end{tikzpicture}
\]
representing a chain of adjunctions $u_! \dashv u^* \dashv u_*$. Specifically, for $F \in \PDC_{\!\!\SA }(\SI)$ these functors are given by the Kan extensions $u_! F = \lan_u F$ and, dually, $u_* F = \ran_u F$, notation as in \cite[X.3]{currante}. It follows from the definition of Kan extensions that the adjunctions hold.

Let us make these definitions more explicit with the help of co/limits. 

For $j \in \SJ$, denote by $\SI/j$ the relative slice category. The objects of $\SI/j$ are pairs  $(i, u(i) \to j)$  with $u(i) \to j$ a morphism in $\SJ$. An arrow from $(i, u(i) \to j)$ to $(i', u(i') \to j)$ is a morphism $i \to i'$ such that $u(i) \to u(i')$ commutes with the structural arrows to $j$. Denote by $\pi_j \colon \SI/j \to \SI$ the forgetful functor that takes $(i, u(i) \to j)$ into $i$ and analogously for arrows.

The functor
\[
u_! \colon \CCC(\SA)^{\SI} \lto \CCC(\SA)^{\SJ}
\]
is defined for $j \in \SJ$ by
\begin{equation} \label{uadmirado}
 u_!(F)(j) := \dirlim{\SI/j} (F \circ \pi_j).
\end{equation}
The functoriality of $u_!(F)$ follows from the existence of the commutative diagram
\[
\begin{tikzpicture}
      \draw[white] (0cm,1.25cm) -- +(0: \linewidth)
      node (G) [black, pos = 0.4] {$\SI/j$}
      node (H) [black, pos = 0.6] {$\SI/j'$};
      \draw[white] (0cm,0.25cm) -- +(0: \linewidth)
      node (E) [black, pos = 0.5] {$\SI$};
      \draw [->] (G) -- (H) node[above, midway, scale=0.75]{$\SI/\beta$};
      \draw [->] (G) -- (E) node[auto, swap, midway, scale=0.75]{$\pi_j$};
      \draw [<-] (E) -- (H) node[auto, swap, midway, scale=0.75]{$\pi_{j'}$};  
\end{tikzpicture}
\]
associated to an arrow $\beta \colon j \to j'$ in $\SJ$, where $\SI/\beta \colon \SI/j \to \SI/j'$ is the functor  that takes an object $(i, u(i) \to j)$ to the composition $(i, u(i) \to j \to j')$, so we have that 
\[
 u_!(F)(\beta) \colon \dirlim{\SI/j} (F \circ \pi_j) \lto
\dirlim{\SI/j'} (F \circ \pi_{j'}) 
\]
is induced by the universal property of colimits. Indeed, notice that it holds that $F \circ \pi_j = F \circ \pi_{j'} \circ \SI/\beta$.

Dually, let us define $j\backslash\SI$, the relative coslice category. Its objects are pairs $(i, j \to u(i))$ with $j \to u(i)$ a morphism in $\SJ$. Its arrows correspond to morphisms $i \to i'$ such that $u(i) \to u(i')$ commutes with the structural maps from $j$. Let $\varpi_j \colon j\backslash\SI \to \SI$ the forgetful functor taking $(i, j \to u(i))$ to $i$ and similarly for arrows.

The functor
\[
u_* \colon \CCC(\SA)^{\SI} \lto \CCC(\SA)^{\SJ}
\]
is defined, for $j \in \SJ$, by
\begin{equation} \label{uestrella}
 u_*(F)(j) := \invlim{j\backslash\SI} (F \circ \varpi_j)
\end{equation}
The functoriality of $u_*(F)$ follows (as in the previous case) from the commutative diagram
\[
\begin{tikzpicture}
      \draw[white] (0cm,1.25cm) -- +(0: \linewidth)
      node (G) [black, pos = 0.4] {$j'\backslash\SI$}
      node (H) [black, pos = 0.6] {$j\backslash\SI$};
      \draw[white] (0cm,0.25cm) -- +(0: \linewidth)
      node (E) [black, pos = 0.5] {$\SI$};
      \draw [->] (G) -- (H) node[above, midway, scale=0.75]{$\beta\backslash\SI$};
      \draw [->] (G) -- (E) node[auto, swap, midway, scale=0.75]{$\varpi_{j'}$};
      \draw [<-] (E) -- (H) node[auto, swap, midway, scale=0.75]{$\varpi_j$};  
\end{tikzpicture}
\]
associated to an arrow $\beta \colon j \to j'$,	 where $\beta\backslash\SI \colon j'\backslash\SI \to j\backslash\SI$ is a functor that takes an object $(i, j' \to u(i))$ to the composition $(i, j \to j' \to u(i))$. The morphism $u_*(F)(\beta)$ is induced by $\beta\backslash\SI$ using the universal property of limits. The construction is dual to the previous one.
\end{cosa}


\section{The basic axioms}\label{dos}

\begin{cosa}
We keep denoting by $\SA$  a Grothendieck category with enough projective objects. From the  category of complexes in $\SA$ we obtain the derived category by inverting quasi-isomorphisms and denote it by $\D(\SA)$. For the basics and notations on these categories we follow \cite{vtc}.  Using this construction, let us introduce another prederivator as a strict 2-functor, the \emph{homological derivator} of $\SA$ as follows:
\[
\PDD_{\!\!\SA } \colon \cat^\op \lto \bigcat
\]
defined over objects by $\PDD_{\!\!\SA }(\SI) := \D(\SA^{\SI})$ and over functors $u \colon \SI \to \SJ$ by taking $\PDD_{\!\!\SA }(u) := u^*$, where the functor $u^*$ is precomposing with $u$ as before. Notice that $u^*$ extends to the corresponding derived categories because $u^*$ is an exact functor. Similarly for natural transformations. 

\end{cosa}

In what follows, we check that $\PDD_{\!\!\SA }$ is actually a derivator. The first two properties are straightforward observations.

\begin{prop}\label{der1}
 The 2-functor $\PDD_{\!\!\SA }$ takes coproducts into products. In other words, the axiom \emph{(Der1)} in \cite{gr13} is satisfied.
\end{prop}

\begin{proof}
The 2-functor $\PDC_{\!\!\SA }$ sends coproducts into products by representability. For any family of small categories $\{\SI_{\lambda}\}_{\lambda \in \Lambda}$, the equivalence 
\[
\CCC\left(\A^{\coprod_{\lambda \in \Lambda} \SI_{\lambda}}\right)  \cong 
\prod_{\lambda \in \Lambda}  \CCC\left(\A^{\SI_{\lambda}}\right)
\]
 preserves quasi-isomorphisms therefore induces an equivalence
 \[
\D\left(\A^{\coprod_{\lambda \in \Lambda} \SI_{\lambda}}\right)  \cong 
\prod_{\lambda \in \Lambda}  \D\left(\A^{\SI_{\lambda}}\right)
\]
as wanted.
\end{proof}

\begin{cosa}\label{funnynot1}
 Let $\SI$ be a small category and $i \in \SI$ an object, otherwise thought as a functor $i \colon \es \to \SI$. It defines a functor $i^* \colon  \PDD_{\!\!\SA }(\SI) \to \PDD_{\!\!\SA }(\es)$. Given $F, G \in \PDD_{\!\!\SA }(\SI)$ and a morphism $f \colon F \to G$ in $\PDD_{\!\!\SA }(\SI)$ denote
$F_i := i^*F$ and $f_i := i^* f$.
\end{cosa}

\begin{prop}\label{der2}
A morphism $f \colon F \to G$ in $\PDD_{\!\!\SA }(\SI)$ is an isomorphism if, and only if $f_i \colon F_i \to G_i$  is an isomorphism in $\PDD_{\!\!\SA }(\es)$ for every $i \in \SI$. In other words, the axiom \emph{(Der2)} in \cite{gr13} is satisfied.
\end{prop}

\begin{proof}
It is clear that if $f$ is an isomorphism, then all the $f_i$ are. Conversely, after taking appropriate resolutions we may represent any $f \colon F \to G$ in $\D(\SA^\SI)$  by an actual morphism of diagrams of complexes. In this case, if all the $f_i$ are quasi-isomorphisms then so is $f$, whence the conclusion. 
\end{proof}

\begin{cosa}
Let  $u \colon \SI \to \SJ$ be a functor of small categories. By virtue of \cite[Proposition 4.3]{AJS}, every object in $\PDC_{\!\!\SA }(\SI)$ has a $q$-projective resolution because $\SA^\SI$ is a Grothendieck category with a projective generator (because so is $\SA$). The existence of $\LL{}u_!$, the left derived functor of $u_!$, follows then from (the dual of) \cite[Proposition 2.2.6.]{yellow}. Similarly $\R{}u_*$, the right derived functor of  $u_*$, exists thanks to the existence of $q$-injective resolutions in $\PDC_{\!\!\SA }(\SI)$, by \cite[Theorem 5.4]{AJS}. Recall that a $\Delta$-adjunction is simply an adjunction where the unit, or equivalently the counit, is a $\Delta$-functor, see \cite[Lemma-Definition 3.3.1]{yellow}
\end{cosa}

\begin{prop}\label{der3}
For a functor $u \colon \SI \to \SJ$, we have a chain of $\Delta$-adjunctions 
\[
\LL{}u_! \dashv u^* \!\dashv \R{}u_*
\] that is depicted in the following diagram:
\[
\begin{tikzpicture}
      \node (G) {$\PDD_{\!\!\SA }(\SI)$};  
      \node[node distance=3cm, right of = G] (H) {$\PDD_{\!\!\SA }(\SJ)$};
      \draw[->, bend left=30] (G) to node [above] {\footnotesize $\LL{}u_!$} (H); 
      \draw[->, bend right=30] (G) to node [below] {\footnotesize $\R{}u_*$} (H);
      \draw[<-,] (G) to node [above] {\footnotesize $u^*$} (H);
\end{tikzpicture}
\]
In other words, the axiom \emph{(Der3)} in \cite{gr13} is satisfied.
\end{prop}

\begin{proof}
It is enough to construct the unit and counit of the adjunction and check that they satisfy the triangle identities \cite[Ch. 4, Theorem 2]{currante}.
To establish the adjunction $\LL{}u_! \dashv u^*$ we need to define the unit and counit
\[
\eta_! \colon \iid \lto u^* \LL{}u_! \quad \varepsilon_! \colon \LL{}u_! u^* \lto \iid 
\]
and check they satisfy the so-called triangle identities. These maps will arise from their underived versions. Indeed, by considering a $q$-projective resolution $P_{u^*F} \iso u^*F$ for $F \in \PDD_{\!\!\SA }(\SJ)$, define the counit map $\varepsilon_!$ as the compostion
\[
\varepsilon_! \colon 
\LL{}u_! u^* F \liso u_! P_{u^*F} \lto u_! u^* F \overset{\varepsilon_{\CCC}}{\lto} F
\]
where $\varepsilon_{\CCC} \colon u_!u^* \to \iid$ denotes the underived counit. Analogously define
\[
\eta_! \colon 
G \losi P_G \overset{\eta_{\CCC}}\lto u^* u_! P_G \liso u^*\LL{}u_! G
\]
where $\eta_{\CCC} \colon \iid \to u^*u_!$ denotes the underived unit. Recall that $u^*$ preserves quasi-isomorphisms.

We have to check that the following compositions 
\begin{align}
u^* \xto{\eta u^*}\,      &u^* \LL{}u_! u^*   \xto{u^* \varepsilon} u^* 
\label{triuno}\\
\LL{}u_! \xto{\LL{}u_! \eta}\, &\LL{}u_! u^* \LL{}u_! \xto{\LL{}u_! \varepsilon} \LL{}u_! \label{tridos}
\end{align}
are identities. For this matter, let $F \in \PDD_{\!\!\SA }(\SJ)$ and $G \in \PDD_{\!\!\SA }(\SI)$. Choose $q$-projective resolutions $P_F \iso F$ and $P_G \iso G$. We will see first that the composition \eqref{triuno} is the identity. Let $P_{u^*F} \iso u^*F$ be a $q$-projective resolution. Consider the commutative diagram
\[
\begin{tikzpicture}
\matrix (m) [matrix of math nodes,
row sep=2.5em, column sep=3em,
text height=1.5ex, text depth=0.25ex]{
u^*F      & u^*u_!u^*F      & u^*F \\
u^*P_F    & u^*u_!P_{u^*F}  &  \\
u^*F      & u^*\LL{}u_!u^*F & u^*F \\          
};
\path[->,font=\scriptsize,>=angle 90]
(m-1-1) edge node[above] {via $\eta_{\CCC}$} (m-1-2)  
(m-2-1) edge node[above] {via $\eta_{\CCC}$} (m-2-2)
        edge node[left] {$\wr$} (m-1-1)
        edge node[left] {$\wr$} (m-3-1)
(m-2-2) edge node[left] {$\wr$} (m-3-2)
        edge node[left] {$can$} (m-1-2)
(m-1-2) edge node[above] {via $\varepsilon_{\CCC}$} (m-1-3)
(m-3-1) edge node[below] {via $\eta_!$} (m-3-2)
(m-3-2) edge node[below] {via $\varepsilon_!$} (m-3-3);
\draw [double] (m-1-3) to (m-3-3);
\draw [double] (m-1-1) to[bend right=45] (m-3-1);
\end{tikzpicture}
\]
The top row is the identity corresponding to the underived adjunction thus the bottom, which is \eqref{triuno} applied to $F$, also is.

For \eqref{tridos}, let $P_G \liso G$ be a $q$-projective resolution.  The diagram
\[
\begin{tikzpicture}
\matrix (m) [matrix of math nodes,
row sep=2.5em, column sep=3em,
text height=1.5ex, text depth=0.25ex]{
u_!P_G     & u_!u^*u_!P_G           & u_!P_G \\
           & u_!P_{u^*u_!P_G}       &  \\
\LL{}u_!G  & \LL{}u_!(u^*\LL{}u_!G) & \LL{}u_!G \\          
};
\path[->,font=\scriptsize,>=angle 90]
(m-1-1) edge node[above] {via $\eta_{\CCC}$} (m-1-2)
        edge node[left] {$\wr$} (m-3-1)
(m-1-2) edge node[above] {via $\varepsilon_{\CCC}$} (m-1-3)
(m-2-2) edge node[left] {$can$} (m-1-2)
        edge node[left] {$\wr$} node[right] {$\rho$} (m-3-2)
(m-1-3) edge node[right] {$\wr$} (m-3-3)
(m-3-1) edge node[below] {via $\eta_!$} (m-3-2)
(m-3-2) edge node[below] {via $\varepsilon_!$} (m-3-3);
\end{tikzpicture}
\]
is commutative. The vertical quasi-isomorphism $\rho$ is the composition 
\[
u_!P_{u^* u_!P_G} \liso \LL{}u_!(u^* u_!P_G) \liso \LL{}u_!(u^* \LL{}u_!G)
\]
where the first isomorphism is defined through a $q$-projective resolution $P_{u^* u_! P_G} \iso u^* u_!P_G$. It follows that the top row is the identity and so is the bottom one.

Let us change the side of the structure. We have to construct adjunction $u^* \dashv \R{}u_*$ and we will do it by giving its unit and counit. We want thus to define natural morphisms
\[
\eta_* \colon \iid \lto \R{}u_*u^* \quad \varepsilon_* \colon u^*\R{}u_* \lto \iid 
\]
Let $G \iso \SI_G$ be a $q$-injective resolution of $G \in \PDD_{\!\!\SA }(\SI)$. Let $\varepsilon_*$ be the composite
\[
\varepsilon_* \colon 
u^*\R{}u_*G \liso u^*u_*\SI_G \overset{\varepsilon_{\CCC}}{\lto} \SI_G \losi G
\]
where $\varepsilon_{\CCC} \colon u_*u^* \to \iid$ denotes the underived counit. Let $F \in \PDD_{\!\!\SA }(\SJ)$ and $u^*F \iso \SI_{u^*F}$ a $q$-injective resolution.
The composed morphism 
\[
\eta_* \colon
F \overset{\eta_{\CCC}}\lto u_*u^*F  \liso u_*\SI_{u^*F} \liso \R{}u_*u^* F
\]
defines the natural transformation $\eta_*$.

Now we have to check that the following compositions 
\begin{align}
u^* \xto{u^*\eta}\, &u^* \R{}u_*u^* \xto{\varepsilon u^*} u^* \label{tritres}\\
\R{}u_* \xto{\eta\R{}u_*}\, &\R{}u_*u^*\R{}u_* \xto{\R{}u_*\varepsilon} \R{}u_* \label{tricuatro}
\end{align}
are the identity. 
This fact follows from arguments dual to the ones used for the 
 identities \eqref{triuno} and \eqref{tridos}. We leave the details to the reader.

Finally, by \cite[Lemma-Definition 3.3.1]{yellow} it is enough to recall that the unit (or counit) of the purported adjunctions are $\Delta$-functorial to ensure the adjunctions are isomorphisms of $\Delta$-functors.
\end{proof}

\begin{cosa}\label{funnynot2}
 For the canonical functor $c \colon \SI \to \es$ note that
\[
\LL c_! = \LL\shdirlim{} \quad \R{}c_* = \R\shinvlim{}.
\]
Now Proposition \ref{der3} is expressed as the following chain of $\Delta$-functorial adjunctions
\[
\begin{tikzpicture}
      \node (G) {$\PDD_{\!\!\SA }(\SI)$};  
      \node[node distance=3cm, right of = G] (H) {$\PDD_{\!\!\SA }(\es)$};
      \draw[->, bend left=30] (G) to node [above] {\footnotesize $\LL\shdirlim{}$} (H);
      \draw[->, bend right=30] (G) to node [below] {\footnotesize $\R\shinvlim{}$} (H);
      \draw[<-,] (G) to node [above] {\footnotesize $c^*$} (H);
\end{tikzpicture}
\]
 Notice that in this case $\PDD_{\!\!\SA }(\es) = \D(\SA)$. The functors $\LL\shdirlim{}$ and $\R\shinvlim{}$ play in the derived category the roles of homotopy colimit and homotopy limit, respectively.
\end{cosa}

\begin{cosa} \textbf{The Beck-Chevalley maps.}
 Consider a commutative square (up to isomorphism) of functors:
\begin{equation}\label{cuadradito}
\begin{tikzpicture}[baseline=(current  bounding  box.center)]
\matrix (m) [matrix of math nodes,
row sep=3em, column sep=3em,
text height=1.5ex, text depth=0.25ex]{
\SI'  & \SJ' \\
\SI   & \SJ \\          
};
\path[->,font=\scriptsize,>=angle 90]
(m-1-1) edge node[above] {$u'$} (m-1-2)
        edge node[left]  {$v'$} (m-2-1)
(m-1-2) edge node[right] {$v$} (m-2-2)
(m-2-1) edge node[below] {$u$} (m-2-2);
\end{tikzpicture}
\end{equation}
From the natural isomorphism $v'^* u^* \iso u'^* v^*$ we obtain morphisms
\begin{align*}
u'_! v'^* \lto v^* u_! \\ v^* u_* \lto u'_* v'^* 
\end{align*}
defined by composing the following strings of natural maps, respectively,
\begin{align*}
u'_! v'^* &\xto{\text{ via }\eta} u'_! v'^* u^* u_! \liso
u'_! u'^* v^* u_! \xto{\text{ via }\varepsilon}  v^* u_! \\
v^* u_* &\xto{\text{ via }\eta} u'_* u'^* v^* u_* \liso
u'_* v'^* u^* u_* \xto{\text{ via }\varepsilon} u'_* v'^* 
\end{align*}
These maps are discussed in detail in \cite[Section 1.2]{gr13} as a ``calculus of mates''.

The next proposition gives conditions for these morphisms to be isomorphisms. We begin with a previous lemma.
\end{cosa}

\begin{lem}\label{qconserv}
 For a functor $u \colon \SI \to \SJ$, and $j \in \SJ$, denote by $\pi_j \colon \SI/j \to \SI$ and $\varpi_j \colon j\backslash\SI \to \SI$ the canonical functors considered in \ref{expdesadj}. Then, $\pi_j^*$ preserves $q$-projective objects and, dually, $\varpi_j^*$ preserves $q$-injective objects.
\end{lem}

\begin{proof}
We need to show first that $\pi_j^*$ preserves $q$-projective objects. Since we have the adjunction $\pi_j^* \dashv \pi_{j *}$ this will follow from the fact that $\pi_{j *}$ is an exact functor. It is enough to check it element-wise. Notice first that, for every $F \in \PDC_{\!\!\SA }(\SI/j)$ and every $\mathbf{i} \in \SI/j$ with $\mathbf{i} = (i, u(i) \to j)$,
\[
\pi_{j *}(F)(\mathbf{i}) \ = \invlim{i \backslash (\SI/j)} (F \circ \varpi'_i)
\]
where $\varpi'_i \colon i \backslash (\SI/j) \to \SI/j$ is the canonical functor.
The discrete fiber $\bar\pi_j^{-1}(i)$, formed by those objects in the preimage of $i$ with only identities as morphisms, is identified with a cofinal subcategory of $i \backslash (\SI/j)$. To see this, let $\sigma \colon \bar\pi_j^{-1}(i) \to i \backslash (\SI/j)$ be the fully faithful functor defined on objects by 
\[
(i, u(i) \overset{\beta}\to j) \rightsquigarrow 
(i, u(i) \overset{\beta}\to j, i \overset{\id}\to i).
\]
The subcategory $\bar\pi_j^{-1}(i)$ is cofinal because the functor $\sigma$ possesses a right adjoint, namely $\sigma' \colon i \backslash (\SI/j) \to  \bar\pi_j^{-1}(i)$ defined on objects by 
\[(i', u(i') \overset{\beta'}\to j, i \overset{\gamma}\to i') \rightsquigarrow 
(i, u(i) \overset{u(\gamma)}\to u(i') \overset{\beta'}\to j).\]
Notice that $\sigma' \sigma = \id_{\bar\pi_j^{-1}(i)}$. In other words, the functor $\sigma$ is cofinal, see \cite[Definition 2.5.1]{ks}. With this in mind, we have that, for any morphism $v$ in $\PDC_{\!\!\SA }(\SI/j)$,
\[
\invlim{i \backslash (\SI/j)} v \circ \varpi'_i \cong
\invlim{\bar\pi_j^{-1}(i)} v \circ \varpi'_i \circ \sigma =
\prod_{\bar\pi_j^{-1}(i)} v(\mathbf{i})
\]
and products are exact in $\CCC(\SA)$, because $\SA$ has projective generators.

Now we need that $\varpi_j^*$ preserves $q$-injective objects. By the adjunction $\varpi_{j !} \dashv \varpi_j^*$ it suffices that $\varpi_{j !}$ is exact. As before, we argue element-wise. We have for every $F \in \PDC_{\!\!\SA }(j\backslash\SI)$ and every $\mathbf{i} \in j\backslash\SI$,
\[
\varpi_{j !}(F)(\mathbf{i}) \ = \dirlim{(j\backslash\SI)/ i} (F \circ \pi'_i)
\]
where $\pi'_i \colon (j\backslash\SI)/ i \to j\backslash\SI$ is the canonical functor. 
Arguing dually as before we see that the discrete fiber $\bar\varpi_j^{-1}(i)$ is final\footnote{Or \emph{co-cofinal} in the terminology of \cite[Definition 2.5.1]{ks}.} in $(j\backslash\SI)/ i$ through the final functor $\varsigma \colon \bar\varpi_j^{-1}(i) \to (j\backslash\SI)/ i$ defined in a dual fashion to $\sigma$. Let $w$ be a morphism in $\PDC_{\!\!\SA }(j\backslash\SI)$. We have the isomorphisms
\[
\dirlim{(j\backslash\SI)/ i} w \circ \pi'_i \cong
\dirlim{\bar\varpi_j^{-1}(i)} w \circ \pi'_i \circ \varsigma =
\coprod_{\bar\varpi_j^{-1}(i)} w(\mathbf{i})
\]
and dually as before, we conclude because coproducts are exact in $\CCC(\SA)$.
\end{proof}

\begin{rem}
 The argument is adapted from the proof of Der4 in \cite[Theorem 6.1.7]{bld} but the context here is somewhat simpler.
\end{rem}

\begin{prop}\label{der4}
For a functor $u \colon \SI \to \SJ$, and $j \in \SJ$, consider the commutative diagrams of functors
\[
\begin{tikzpicture}
\matrix (m) [matrix of math nodes,
row sep=3em, column sep=3em,
text height=1.5ex, text depth=0.25ex]{
\SI/j  & \es \\
\SI    & \SJ \\          
};
\path[->,font=\scriptsize,>=angle 90]
(m-1-1) edge node[above] {$p$} (m-1-2)
        edge node[left] {$\pi_j$} (m-2-1)
(m-1-2) edge node[right] {$j$} (m-2-2)
(m-2-1) edge node[below] {$u$} (m-2-2);
\end{tikzpicture}
\qquad
\begin{tikzpicture}
\matrix (m) [matrix of math nodes,
row sep=3em, column sep=3em,
text height=1.5ex, text depth=0.25ex]{
j\backslash\SI  & \es \\
\SI          & \SJ \\          
};
\path[->,font=\scriptsize,>=angle 90]
(m-1-1) edge node[above] {$q$} (m-1-2)
        edge node[left] {$\varpi_j$} (m-2-1)
(m-1-2) edge node[right] {$j$} (m-2-2)
(m-2-1) edge node[below] {$u$} (m-2-2);
\end{tikzpicture}
\]
The corresponding Beck-Chevalley natural transformations 
\begin{align}
 \LL p_!\pi_j^* \lto j^* \LL u_! \label{BCDad}\\
 j^*\R{}u_* \lto \R{}q_*\varpi_j^*\label{BCDest}
\end{align}
are isomorphisms. In other words, the axiom \emph{(Der4)} in \cite{gr13} is satisfied.
\end{prop}

\begin{proof}
The fact that $\PDC_{\!\!\SA }$ is a derivator \eqref{derC} provides us the following isomorphisms
\begin{align}
 p_!\pi_j^* &\liso j^* u_!        \label{BCCad}\\ 
 j^* u_*    &\liso q_*\varpi_j^*. \label{BCCest}
\end{align}
Let us check \eqref{BCDad} first. Let $F \in \PDD_{\!\!\SA }(\SI)$ and let $P_F \iso F$ be a $q$-projective resolution. We have that
\[
\LL p_!\pi_j^* (F) \cong \LL p_!\pi_j^* (P_F) \cong p_!\pi_j^* (P_F)
\]
where the last canonical isomorphism is due to the fact that $\pi_j^* (P_F)$ is also $q$-projective by Lemma \ref{qconserv}. Furthermore,
\[
j^*\LL u_!(F) \cong j^* \LL u_! (P_F) \cong j^* u_! (P_F).
\]
By the isomorphism \eqref{BCCad}, it holds that $p_!\pi_j^* (P_F) \cong j^* u_! (P_F)$ so \eqref{BCDad} is established.

Let us check now \eqref{BCDest}. Let $F \in \PDD_{\!\!\SA }(\SI)$ and $F \iso I_F$ be a $q$-injective resolution. We have
\[
 j^*\R{}u_*(F) \cong  j^*\R{}u_* (I_F) \cong  j^*u_* (I_F)
\]
On the other hand
\[
\R{}q_*\varpi_j^* (F) \cong \R{}q_*\varpi_j^* (I_F) \cong q_*\varpi_j^* (I_F)
\]
where the last canonical isomorphism is due to the fact that $\varpi_j^* (I_F)$ is also $q$-injective by Lemma \ref{qconserv}, dual case. Now we can use the isomorphism \eqref{BCCest}, $j^*u_* (I_F) \cong q_*\varpi_j^* (I_F)$ and we conclude.
\end{proof}

\begin{rem}
Let  $F \in \PDD_{\!\!\SA }(\SI)$, using the notations in \ref{funnynot1}, $j^*\LL u_!(F) \cong (\LL u_!F)_j$ and $j^*\R{}u_*(F) \cong (\R{}u_*F)_j$. Also, in the setting of \ref{funnynot2}, the functors $p$, $q$ are identified with the canonical functor $c$. With this notations, formulas \eqref{BCDad} and \eqref{BCDest} become
 \[
 \LL\shdirlim\, \pi_j^*(F) \liso (\LL u_!F)_j \quad \text{ and } \quad
 (\R{}u_*F)_j \liso \R\shinvlim\, \varpi_j^*(F)
 \]
respectively.  These formulas may be expressed concisely as the fact that derived co/limits are computed pointwise.
\end{rem}
 
\section{The strong derivator of derived categories}\label{tres}

\begin{cosa} \textbf{The underlying diagram functor.}
Let us consider, for two small categories $\SI$ and $\SJ$, the canonical isomorphism
\[
\PDC_{\!\!\SA }(\SI \times \SJ) = \CCC(\SA^{\SI \times \SJ}) \liso \CCC(\SA^{\SJ})^\SI = \PDC_{\!\!\SA }(\SJ)^\SI
\]
Denote the composed isomorphism as
\[
\dia_{\SI, \SJ} \colon \PDC_{\!\!\SA }(\SI \times \SJ) \lto \PDC_{\!\!\SA }(\SJ)^\SI
\]
 It yields the following diagram
\begin{equation}\label{underdia}
\begin{tikzpicture} [baseline=(current  bounding  box.center)]
\matrix (m) [matrix of math nodes,
row sep=3em, column sep=3em,
text height=1.5ex, text depth=0.25ex]{
\PDC_{\!\!\SA }({\SI \times \SJ}) & \PDC_{\!\!\SA }({\SJ})^\SI \\
\PDD_{\!\!\SA }({\SI \times \SJ}) & \PDD_{\!\!\SA }({\SJ})^\SI \\          
};
\path[->,font=\scriptsize,>=angle 90]
(m-1-1) edge node[above] {$\dia_{\SI, \SJ}$} (m-1-2)
        edge node[left] {$Q$} (m-2-1)
(m-1-2) edge node[right] {$Q^{\SI}$} (m-2-2)
(m-2-1) edge node[below] {$\ddia_{\SI, \SJ}$} (m-2-2);
\end{tikzpicture}
\end{equation}
where by $Q, Q^{\SI}$ we denote the canonical quotient functors. The existence of $\ddia_{\SI, \SJ}$ is guaranteed by the fact that $Q^{\SI} \circ \dia_{\SI, \SJ}$ takes quasi-isomorphisms into isomorphisms. Notice that $\ddia_{\SI, \SJ}$ need not be an equivalence of categories. The following proposition gives the next best property about this functor.
\end{cosa}

\begin{cosa}
 We will denote by $\dos$ the category corresponding to the poset $\{0 \to 1\}$ sometimes called the ``walking arrow''. Notice that for an arbitrary category $\SC$, the functor category $\SC^{\dos}$ is the arrow category of $\SC$. The category $\dos$ is denoted $[1]$ in \cite{gr13}. We have opted for a different notation to avoid a conflict with the notation for the shift functor in complex categories.
\end{cosa}

\begin{prop}\label{der5}
The underlying functor 
\[
 \ddia_{\dos, \SJ} \colon \PDD_{\!\!\SA }(\dos \times {\SJ}) \to \PDD_{\!\!\SA }({\SJ})^{\dos}
\]
is full and essentially surjective for each category $\SJ$. In other words, the derivator $\PDD_{\!\!\SA }$ is strong, otherwise said, the axiom \emph{(Der5)} in \cite{gr13} is satisfied.
\end{prop}

\begin{proof}
 Let us see that it is essentially surjective. Let $[\phi \colon F \to G]$ be an object of $\PDD_{\!\!\SA }({\SJ})^{\dos}$. It corresponds to a diagram 
 \[
F \xot{\sigma} F' \xto{\psi}  G
 \]
 in $\PDC_{\!\!\SA }({\SJ})$ with $\sigma$ a quasi-isomorphism. In other words, we represent $\phi$ as the fraction $\psi / \sigma$. By taking a $q$-injective resolution $G \to I_G$, it induces a map of complexes $\phi' \colon F \to I_G$ such that $\phi = Q^{\dos}(\phi')$.  By the equivalence of categories $\dia_{\dos, \SJ}$ there is a 
 $\phi'' \in \PDC_{\!\!\SA }({\dos \times \SJ})$ whose image is $\phi'$. Now, by the commutativity of the square \eqref{underdia} the image under $\ddia_{\dos, \SJ}$ of $Q(\phi'')$ is $\phi$.
 
To complete the proof, we have to prove that $\ddia_{\dos, \SJ}$ is surjective on maps. After taking $q$-injective resolutions, any object in $\PDD_{\!\!\SA }({\SJ})^{\dos}$ is represented by map of complexes. Let $[\varphi] = [\varphi \colon F \to G]$ be such an object. A morphism in $\PDD_{\!\!\SA }({\SJ})^{\dos}$, $\ggmm \colon [\varphi] \to [\varphi']$ is represented by a square of morphisms of complexes
\begin{equation}\label{maphot}
\begin{tikzpicture} [baseline=(current  bounding  box.center)]
\matrix (m) [matrix of math nodes,
row sep=3em, column sep=3em,
text height=1.5ex, text depth=0.25ex]{
F  & G \\
F' & G' \\          
};
\path[->,font=\scriptsize,>=angle 90]
(m-1-1) edge node[above] {$\varphi$} (m-1-2)
        edge node[left] {$\gamma_1$} (m-2-1)
(m-1-2) edge node[right] {$\gamma_2$} (m-2-2)
(m-2-1) edge node[below] {$\varphi'$} (m-2-2);
\draw [shorten >=0.2cm,shorten <=0.2cm, ->, dashed] (m-1-1) -- (m-2-2) node[auto, midway,font=\scriptsize]{$s$};
\end{tikzpicture}
\end{equation}
that commutes \emph{up to homotopy}. We specify the homotopy, a graded map $s \colon F[1] \to G'$ such that $\gamma_2 \varphi - \varphi' \gamma_1 = d s + s d$. We denote $\ggmm = (\gamma_1, \gamma_2, s)$. To get a preimage through $\ddia_{\dos, \SJ}$ for $\ggmm$ it is enough to find a quasi-isomorphic square that commutes in $\PDC_{\!\!\SA }({\SJ})$. We will show that this object is lifted to $\PDC_{\!\!\SA }({\SJ})^{\dos} \cong \PDC_{\!\!\SA }({\dos \times \SJ})$ and it will provide the desired preimage in $\PDD_{\!\!\SA }({\dos \times \SJ})$.

Consider also the following diagram 
\begin{equation}\label{mapiso}
\begin{tikzpicture} [baseline=(current  bounding  box.center)]
\matrix (m) [matrix of math nodes,
row sep=3em, column sep=3em,
text height=1.5ex, text depth=0.25ex]{
F  & G \\
F  & \cyl(\varphi) \\          
};
\path[->,font=\scriptsize,>=angle 90]
(m-1-1) edge node[above] {$\varphi$} (m-1-2)
        edge node[left] {$\id$} (m-2-1)
(m-1-2) edge node[right] {$\alpha$} (m-2-2)
(m-2-1) edge node[below] {$\iota$} (m-2-2);
\draw [shorten >=0.2cm,shorten <=0.2cm, ->, dashed] (m-1-1) -- (m-2-2) node[auto, midway,font=\scriptsize]{$t$};
\end{tikzpicture}
\end{equation}
where, as usual, for $n \in \ZZ$,
\[
\cyl(\varphi)^n = F^{n+1} \oplus F^n \oplus G^n \quad \text{ with differential } \quad
\footnotesize{
\begin{pmatrix}
-d       & 0 & 0 \\
\id      & d & 0 \\
-\varphi & 0 & d
\end{pmatrix}},
\]
the morphism $\iota$ is the inclusion in the second component and $\alpha$ is the inclusion in the third component. This diagram is commutative up to homotopy. Indeed, let the graded map $t \colon F[1] \to \cyl(\varphi)$ be defined as the inclusion in the first component. It holds that $\iota - \alpha \varphi = d t + t d$, verifying the claim.

The morphism $\alpha$ has a homotopy inverse, namely $\beta:= (0~ \varphi~ \id)$. In fact $\beta \alpha = \id$ while $\alpha \beta$ is just homotopic to $\id$ through the homotopy given by the matrix
\[ 
{\footnotesize
\begin{pmatrix}
0 & \id & 0 \\
0 & 0   & 0 \\
0 & 0   & 0
\end{pmatrix}}
\]
It holds $\cyl(\varphi) \cong G$ in the derived category, therefore we get an isomorphism 
$\aall \colon [\varphi] \iso [\iota]$ in $\PDD_{\!\!\SA }({\SJ})^{\dos}$ given by $\aall = (\id, \alpha, 0)$.

We promised to replace \eqref{maphot} by a commutative diagram in $\PDC_{\!\!\SA }({\SJ})$.  
Let $\tilde{\gamma}_2 \colon \cyl(\varphi) \to G'$ be defined by the matrix 
\[(-s\quad \gamma_2\varphi-d_{G'}s-sd_F\quad \gamma_2).\]
 The morphism $\tilde{\gamma}_2$ is a map of complexes because
\begin{align*}
 d_{G'} \tilde{\gamma}_2 &= 
 (-d_{G'}s \quad d_{G'}\gamma_2\varphi - d_G s d_F\quad d_{G'}\gamma_2) =\\
& = (-d_{G'}s\quad \gamma_2\varphi d_{F} - d_G s d_F\quad \gamma_2 d_{G'}) 
\,\,\, = \tilde{\gamma}_2 d_{\cyl(\varphi)}.
\end{align*}
The square
 \begin{equation}\label{maprect}
\begin{tikzpicture} [baseline=(current  bounding  box.center)]
\matrix (m) [matrix of math nodes,
row sep=3em, column sep=3em,
text height=1.5ex, text depth=0.25ex]{
F  & \cyl(\varphi) \\
F' & G' \\          
};
\path[->,font=\scriptsize,>=angle 90]
(m-1-1) edge node[above] {$\iota$} (m-1-2)
        edge node[left] {$\gamma_1$} (m-2-1)
(m-1-2) edge node[right] {$\tilde{\gamma}_2$} (m-2-2)
(m-2-1) edge node[below] {$\varphi'$} (m-2-2);
\end{tikzpicture}
\end{equation}
is commutative \emph{already in} $\PDC_{\!\!\SA }({\SJ})$. This provides a morphism $\tilde{\ggmm} \colon [\iota] \to [\varphi']$ given by $\tilde{\ggmm} = (\gamma_1, \tilde{\gamma}_2, 0)$. It remains to check that $\ggmm = \tilde{\ggmm} \circ \aall$ in $\PDD_{\!\!\SA }({\SJ})^{\dos}$. This amounts to check that $\varphi' \gamma_1$ is homotopic to $\tilde{\gamma}_2 \alpha \varphi$. Notice that 
\begin{align*}
 \varphi' \gamma_1 &= \gamma_2 \varphi - d_{G'}s - s d_F\\
 \tilde{\gamma}_2 \alpha \varphi & = \gamma_2 \varphi
\end{align*}
whose difference is null by the homotopy $s$, so we conclude.

The fact that $\aall$ is an isomorphism in $\PDD_{\!\!\SA }({\SJ})^{\dos}$ implies that $\ggmm$ and $\tilde{\ggmm}$ are isomorphic. Otherwise said, the squares \eqref{maphot} and \eqref{maprect} correspond to the same object in $\PDD_{\!\!\SA }({\SJ})^{\dos}$. But $\tilde{\ggmm}$, which is represented by the former square, has a preimage in $\PDD_{\!\!\SA }({\dos \times \SJ})$ because it is a commutative square of complexes not just up to homotopy.
\end{proof}

\begin{thm}
Let $\SA$ be a Grothendieck category with enough projective objects. The prederivator $\PDD_{\!\!\SA }$ is a strong derivator that is defined over all the 2-category of diagrams $\cat$.
\end{thm}

\begin{proof}
It follows from Propositions \ref{der1}, \ref{der2}, \ref{der3}, \ref{der4} and \ref{der5}.
\end{proof}

\section{Stability}\label{cuatro}

We begin with a trivial observation.

\begin{prop}\label{derpoint}
 Let $\SA$ be a Grothendieck category with enough projective objects. The strong derivator  $\PDD_{\!\!\SA }$ is pointed.
\end{prop}

\begin{proof}
It is clear that $\PDD_{\!\!\SA }(\es)$ (and, in fact, all $\PDD_{\!\!\SA }(\SI)$) are pointed, being additive categories.
\end{proof}

We will look now at the stability property of $\PDD_{\!\!\SA }$. This property is interesting because, as it is well-known, stability is a \emph{property} of derivators while to be a triangulated category is a \emph{structure} we impose. The following discussion expresses the important fact that the usual triangulation defined via mapping cones agrees with the natural triangulation that arises from the derivator.

\begin{cosa} \textbf{Co/cartesian squares.}
We recall here a few basic notations from \cite{gr13}. Let $\square$ be the small category given by the following poset
\[
\begin{tikzpicture} [baseline=(current  bounding  box.center)]
\matrix (m) [matrix of math nodes,
row sep=2em, column sep=2em,
text height=1.5ex, text depth=0.25ex]{
(0,0)  & (1,0) \\
(0,1)  & (1,1) \\          
};
\path[->,font=\scriptsize,>=angle 90]
(m-1-1) edge (m-1-2)
        edge (m-2-1)
(m-1-2) edge (m-2-2)
(m-2-1) edge (m-2-2);
\end{tikzpicture}
\]
Let the following inclusions of subcategories $i_{\ul} \colon \ul \to \square$ and 
$i_\lr \colon \lr \to \square$ be given, respectively, by the subposets: 
\[
\begin{tikzpicture} [baseline=(current  bounding  box.center)]
\matrix (m) [matrix of math nodes,
row sep=2em, column sep=2em,
text height=1.5ex, text depth=0.25ex]{
(0,0)  & (1,0) \\
(0,1)  & \\          
};
\path[->,font=\scriptsize,>=angle 90]
(m-1-1) edge (m-1-2)
        edge (m-2-1);
\end{tikzpicture}
 \qquad
\begin{tikzpicture} [baseline=(current  bounding  box.center)]
\matrix (m) [matrix of math nodes,
row sep=2em, column sep=2em,
text height=1.5ex, text depth=0.25ex]{
       & (1,0) \\
(0,1)  & (1,1) \\          
};
\path[->,font=\scriptsize,>=angle 90]
(m-1-2) edge (m-2-2)
(m-2-1) edge (m-2-2);
\end{tikzpicture}
\]
We say that an object in $\PDD_{\!\!\SA }(\square)$ is \emph{cocartesian} if it lies in the essential image of $\LL{}i_{\ul !} \colon \PDD_{\!\!\SA }(\ul) \to \PDD_{\!\!\SA }(\square)$. Dually, it is called \emph{cartesian} if it lies in the essential image of $\R{}i_{\lr *} \colon \PDD_{\!\!\SA }(\lr) \to \PDD_{\!\!\SA }(\square)$.
\end{cosa}

\begin{lem}\label{cocarttridis}
Let $F \in \PDD_{\!\!\SA }(\square)$. The object $F$ is cocartesian if, and only if, there is a representative in $\PDC_{\!\!\SA }(\square)$ denoted by 
\begin{equation}\label{cocartsq}
 \begin{tikzpicture}[baseline=(current  bounding  box.center)]
\matrix (m) [matrix of math nodes,
row sep=3em, column sep=3em,
text height=2ex,
text depth=0.25ex]{
	F_{00} & F_{10}  \\
	F_{01} & F_{11}  \\
};
\path[->, font=\scriptsize]
(m-1-1) edge node[above] {$\varphi_{E}$} (m-1-2)
        edge node[left] {$\varphi_{S}$} (m-2-1)
(m-1-2) edge node[right] {} (m-2-2)
(m-2-1) edge node[below] {} (m-2-2);
\end{tikzpicture}  
\end{equation}
such that there is distinguished triangle in $\PDD_{\!\!\SA }(\es)$
\begin{equation}\label{cocarttri}
F_{00} \overset{\alpha}\lto F_{10} \oplus F_{01} \lto F_{11} \overset{+}\lto 
\end{equation}
with $\alpha := (\varphi_{E} \quad - \varphi_{S})^{\ts}$.
\end{lem}

\begin{proof}
If $F$ is cocartesian, we have that $F = \LL{}i_{\ul !}F'$ where $F'$ denotes the subdiagram $F' := (F_{00}, F_{10}, F_{01}, \varphi_{E}, \varphi_{S})$. By taking a $q$-projective resolution of $F'$, we may asume that $F = i_{\ul !}F'$, in other words $F_{11} = F_{10} \coprod_{F_{00}} F_{01}$ in $\CCC(\SA)$. In this case we have an exact sequence in $\CCC(\SA)$:
\[
0 \lto F_{00} \overset{\alpha}\lto F_{10} \oplus F_{01} 
  \lto F_{10} \amalg_{F_{00}} F_{01} \lto 0.
\]
This exact sequence gives a distinguished triangle in $\D(\SA)$:
\[
F_{00} \overset{\alpha}\lto F_{10} \oplus F_{01} 
                       \lto F_{10} \amalg_{F_{00}} F_{01}
            \overset{+}\lto .
\]

Conversely, the existence of a commutative square \eqref{cocartsq} (where we will assume that all complexes are $q$-projective), such that the triangle \eqref{cocarttri} is distinguished,  expresses the fact that there is an isomorphism $F_{11} \iso \SC(\alpha)$ in $\PDD_{\!\!\SA}(\es)$, where $\SC(\alpha)$ denotes the usual mapping cone of the morphisms of complexes $\alpha$. On the other hand $F_{10} \amalg_{F_{00}} F_{01}$ is the cokernel of $\alpha$, thus isomorphic to $F_{11}$ which implies that $F$ is cocartesian.
\end{proof}

We have also a dual statement.

\begin{lem}\label{carttridis}
Let $G \in \PDD_{\!\!\SA }(\square)$. The object $G$ is cartesian if, and only if, there is a representative in $\PDC_{\!\!\SA }(\square)$ that we denote by 
\begin{equation}\label{cartsq}
\begin{tikzpicture}[baseline=(current  bounding  box.center)]
\matrix (m) [matrix of math nodes,
row sep=3em, column sep=3em,
text height=2ex,
text depth=0.25ex]{
	G_{00} & G_{10}  \\
	G_{01} & G_{11}  \\
};
\path[->, font=\scriptsize]
(m-1-1) edge node[above] {} (m-1-2)
        edge node[left] {} (m-2-1)
(m-1-2) edge node[right] {$\psi_{S}$} (m-2-2)
(m-2-1) edge node[below] {$\psi_{E}$} (m-2-2);
\end{tikzpicture}  
\end{equation}
such that there is distinguished triangle in $\PDD_{\!\!\SA }(\es)$
\begin{equation}\label{carttri}
G_{00} \lto G_{10} \oplus G_{01} \overset{\beta}\lto  G_{11} \overset{+}\lto 
\end{equation}
with $\beta := (\psi_{S} \quad \psi_{E})$. 
\end{lem}

\begin{proof}
 It is similar to the previous one. If $G$ is cartesian, we have that $G = \R{}i_{\lr *}G'$ with $G' = (G_{10}, G_{10}, G_{10}, \psi_{S}, \psi_{E})$. Replacing $G'$ by a $q$-injective resolution we assume that $G = i_{\lr *}G'$, in which case we have that $G_{00} = G_{10} \times_{G_{11}} G_{01}$ in $\CCC(\SA)$. Thus we have the following exact sequence
 \[
 0 \lto G_{10} \times_{G_{11}} G_{01} \lto G_{10} \oplus G_{01} 
                       \overset{\beta}\lto  G_{11} \lto 0
\]
in $\CCC(\SA)$ with $\beta =  (\psi_{S} \quad \psi_{E})$. As before, this gives a distinguished triangle in $\D(\SA)$
\[
G_{10} \times_{G_{11}} G_{01} \lto G_{10} \oplus G_{01} 
               \overset{\beta}\lto  G_{11} \overset{+}\lto .
\]

Conversely, the existence of a commutative square formed by $q$-injective complexes \eqref{cartsq}, such that the associated triangle \eqref{cocarttri} is distinguished, yields isomorphisms in $\D(\SA)$
\[
\SC(\beta)[-1] \liso G_{10} \times_{G_{11}} G_{01}.
\]
But $G_{10} \times_{G_{11}} G_{01}$ is the kernel of $\beta$, thus isomorphic to $G_{00}$ which implies that $G$ is cartesian.
\end{proof}


\begin{cosa}
 According to \cite[Definition 4.1]{gr13}, a strong derivator is called \emph{stable} if it is pointed and if it satisfies the following property: a square is cocartesian if and only if it is cartesian. We refer to \emph{loc.~cit.} for further comments on this notion. A square that is both cartesian and cocartesian is called a \emph{bicartesian} square.
\end{cosa}

\begin{thm}\label{main}
 Let $\SA$ be a Grothendieck category with enough projective objects. The strong derivator  $\PDD_{\!\!\SA }$ is stable.
\end{thm}

\begin{proof}
By Proposition \ref{derpoint}, it is pointed. Let $F \in \PDD_{\!\!\SA }(\square)$ be cocartesian. For the sake of clarity denote by $P \to F$ a $q$-projective resolution and $F \to I$ a $q$-injective resolution. There are isomorphims $P_{i j} \iso I_{i j}$ in $\PDD_{\!\!\SA }(\es)$ for every $i, j \in \{0, 1\}$. By Lemma \ref{cocarttridis} the following triangle
\[
P_{00} \overset{\alpha_P}\lto  P_{10} \oplus P_{01} \overset{\beta_P}\lto  P_{11} \overset{+}\lto 
\]
is distinguished. 
There is an isomorphism of triangles
\[
\begin{tikzpicture}
\matrix (m) [matrix of math nodes,
row sep=3em, column sep=3em,
text height=2ex,
text depth=0.25ex]{
P_{00}  & P_{10} \oplus P_{01} & P_{11} & \ \\
I_{00}  & I_{10} \oplus I_{01} & I_{11} & \ \\
};
\path[->, font=\scriptsize]
(m-1-1) edge node [auto] {$\alpha_P$} (m-1-2)
        edge node [auto] {$\wr$}      (m-2-1)
(m-1-2) edge node [auto] {$\beta_P$}  (m-1-3)
        edge node [auto] {$\wr$}      (m-2-2)
(m-1-3) edge node [auto] {+}          (m-1-4)
        edge node [auto] {$\wr$}      (m-2-3)
(m-2-1) edge node [below] {$\alpha_I$} (m-2-2)
(m-2-2) edge node [below] {$\beta_I$} (m-2-3)
(m-2-3) edge node [auto] {+} (m-2-4);
\end{tikzpicture} 
\] 
Therefore the square $F$ is cartesian again by the equivalence provided by Lemma \ref{carttridis}. 

Reversing the argument, assuming that $F$ is cartesian, the lower triangle is distinguished by Lemma \ref{carttridis}. By the depicted isomorphism, so is the upper one. As a consequence, $F$ is cocartesian by Lemma \ref{cocarttridis}, and this completes the proof.
\end{proof}

\begin{cor}
 Let $\SA$ be a Grothendieck category with enough projective objects. The triangulated structure in $\PDD_{\!\!\SA }(\es)$ induced by the property of being stable agrees with the usual triangulation of $\D(\SA)$.
\end{cor}

\begin{proof}
Let $\varphi \colon X \to Y$ be a morphism in $\D(\SA)$. Denote by $[\varphi]$ a preimage in $\D(\SA^{\dos})$ that exists by Proposition \ref{der5}. Let $j \colon \SH \to \ul$ be the canonical inclusion where $\SH$ is the subcategory corresponding to $(0,0) \to (0,1)$. Denote $[\varphi_{\ul}] := j_* [\varphi]$. Notice that the underlying diagram of $[\varphi_{\ul}]$ has the following shape
\[
\begin{tikzpicture} [baseline=(current  bounding  box.center)]
\matrix (m) [matrix of math nodes,
row sep=2em, column sep=2em,
text height=1.5ex, text depth=0.25ex]{
X  & Y \\
0  & \\          
};
\path[->,font=\scriptsize,>=angle 90]
(m-1-1) edge node [auto] {$\varphi$} (m-1-2)
        edge node [auto] {} (m-2-1);
\end{tikzpicture}
\]

According to \cite[Beginning of \S 4.2]{gr13} the third point in the distinguished triangle with base $\varphi$ induced by $\PDD_{\!\!\SA }$ is $(\LL{}i_{\ul !} j_* [\varphi])_{11} = (\LL{}i_{\ul !}[\varphi_{\ul}])_{11}$. In other words, applying Lemma \ref{cocarttridis} to the present situation, we obtain a distinguished triangle
\begin{equation*}
X \overset{\varphi}\lto Y \lto (\LL{}i_{\ul !}[\varphi_{\ul}])_{11} \overset{+}\lto 
\end{equation*}
which is identified with the distinguished triangle
\begin{equation*}
X \overset{\varphi}\lto Y \lto C(\varphi) \overset{+}\lto 
\end{equation*}
 associated to $\varphi$ by the usual triangulation of $\D(\SA) = \PDD_{\!\!\SA }(\es)$.
\end{proof}

\begin{rem} \noindent
\begin{enumerate}
 \item[1.] In the proof of the corollary for the construction of the cone, we might as well have used Lemma \ref{carttridis} for the construction of the (negatively shifted) cone.
 \item[2.] The same property applies to the categories $\PDD_{\!\!\SA }(\SI)$ for $\SI$ a small category by considering the corresponding shifted derivators as in \cite[Proposition 4.3]{gr13}.
\end{enumerate}
\end{rem}

\section{Description of smashing localizations}\label{cinco}

\begin{cosa}
 As an application of the framework given by the homological derivator, we place ourselves in the following situation. Let $A$ be a commutative noetherian ring. We will denote by $A\md$ its category of modules. We will continuously use all along the section the properties of the derivator $\PDD_{\!\!A\md}$.
\end{cosa}

\begin{cosa} \textbf{Localizing subcategories.}
For a triangulated category $\ST$ possessing coproducts, a triangulated subcategory stable for coproducts is called a localizing subcategory.

 According to \cite[Theorem 3.3.]{Nct}, localizing subcategories in $\D(A\md)$ correspond to subsets of $\spec(A)$. A localizing subcategory $\loc$ of $\D(A\md)$ is determined by the residue fields $\kappa(\ip)$ of primes $\ip \in \spec(A)$ that it contains. As a matter of fact, a localizing subcategory generated by a set of objects determines a localizing functor. If $Y \subset \spec(A)$, let $\loc_Y$ be the smallest localizing subcategory of $\D(A\md)$ which contains $\kappa(\ip)$ for all $\ip \in Y$. By \cite[Proposition 5.1]{AJS} there is a localization functor $\ell_Y \colon \D(A\md) \to \D(A\md)$ such that its kernel is $\loc_Y$. By \cite[Lemma 3.5.]{Nct}, the localizing functor $\ell_Y$ is characterized by
\begin{equation}\label{carlocY}
 \ell_Y(\kappa(\ip)) = 
\begin{cases} 
 \kappa(\ip) & \text{if } \ip \notin Y \\ 
           0 & \text{if } \ip \in Y .
\end{cases}
\end{equation}
\end{cosa}

\begin{cosa} \textbf{Smashing localizing subcategories.}
Let $Y$ be a specialization-closed subset, in other words, an arbitrary union of closed subsets of $\spec(A)$. It follows from \cite[Corollary 3.4.]{Nct} that in this case $\loc_Y$ is generated by compact objects, \ie perfect complexes of $A$-modules. The theorem that asserts the existence of $\ell_Y$ does not give a useful, manageable description of it. We will give an explicit description of $\ell_Y$ in terms of perfect complexes and standard functors.

Let us recall some facts from \cite[\S 1]{AJS} that will be convenient in order to achieve our goal. Associated to a localization $(\ell, \eta)$ there is an acyclization\footnote{Note that $\gamma$ is denoted $\ell^a$ in \cite[Proposition 1.6]{AJS}.} $(\gamma, \zeta)$ whose properties are dual to $(\ell, \eta)$ and such that they mutually determine each other. For every $M \in \D(A\md)$ there is a distinguished triangle
\[
\gamma M \xto{\zeta_M} M \xto{\eta_M} \ell M \overset{+}\lto 
\]
For the localizing subcategory $\loc_Y$, having in mind \eqref{carlocY}, the functor $\gamma_Y$ is characterized by
\begin{equation} \label{caracyc}
\gamma_Y(\kappa(\ip)) = 
\begin{cases} 
 \kappa(\ip) & \text{if } \ip \in Y \\ 
           0 & \text{if } \ip \notin Y .
\end{cases} 
\end{equation}
 
 We are going to give an explicit description of $\gamma_Y$. From this, one obtains immediately a description of $\ell_Y$.
\end{cosa}

\begin{cosa} \label{defdelsys}
\textbf{A category associated to a specialization closed subset.} 
The support of an ideal  $\ia \subset A$ is the set of primes in $\spec(A)$ that contain it, in other words, the support of the $A$-module $A/\ia$.
 
 Let us define a small category associated to the closed subsets that make up the specialization closed subset\footnote{A specialization closed subset is an arbitary union of closed subsets.} $Y$ and denote it by $\SI_Y$:
  
\begin{description}
 \item[Objects] Homomorphisms $\bff \colon A^n \to A$ ($n \in \NN$), such that the ideal $\Img(\bff)$ is supported into $Y$. 
 \item[Morphisms] An arrow $\overline{\varphi} \colon \bff \to \bgg$ is defined as a  homomorphism of free modules $\varphi \colon A^n \to A^m$ such that the triangle
\[
\begin{tikzpicture} [baseline=(current  bounding  box.center)]
\matrix (m) [matrix of math nodes,
row sep=2em, column sep=2.5em,
text height=1.5ex, text depth=0.25ex]{
A^n &   & A^m\\
    & A &  \\          
};
\path[->,font=\scriptsize,>=angle 90]
(m-1-1) edge node [auto] {$\varphi$} (m-1-3)
        edge node [auto, swap] {$\bff$} (m-2-2)
(m-1-3) edge node [auto] {$\bgg$} (m-2-2);
\end{tikzpicture}
\]
is commutative. 
\end{description}

It is customary to describe an object $\bff$ giving its values on the canonical basis of $A^n$, that is $\bff := (f_1, \dots, f_n)$. Notice that $\Img(\bff) = \langle f_1, \dots, f_n \rangle$. We will commit a slight abuse of notation and call \emph{the support of} $\bff$ the support of the ideal $\Img(\bff)$.
\end{cosa}

\begin{lem} \label{cofil}
 The category $\SI_Y$ is cofiltered.
\end{lem}

\begin{proof}
 According to \cite[I, D\'efinition 2.7.]{sga41}, to ensure that it is pseudo-cofiltered, we have to check two conditions.
 
\begin{itemize}
 \item[PS 1.] A diagram in $\SI_Y$
\[
\begin{tikzpicture} [baseline=(current  bounding  box.center)]
\matrix (m) [matrix of math nodes,
row sep=1em, column sep=2.5em,
text height=1.5ex, text depth=0.25ex]{
    & \bff &  \\ 
    &      & \bhh \\
    & \bgg &  \\          
};
\path[->,font=\scriptsize,>=angle 90]
(m-1-2) edge node [auto] {$\overline{\varphi}$} (m-2-3)
(m-3-2) edge node [auto, swap] {$\overline{\psi}$} (m-2-3);
\end{tikzpicture}
\]
can be completed to a commutative diagram
\[
\begin{tikzpicture} [baseline=(current  bounding  box.center)]
\matrix (m) [matrix of math nodes,
row sep=1em, column sep=2.5em,
text height=1.5ex, text depth=0.25ex]{
                   & \bff &  \\ 
\bff \otimes \bgg  &      & \bhh \\
                   & \bgg &  \\          
};
\path[->,font=\scriptsize,>=angle 90]
(m-1-2) edge node [auto] {$\overline{\varphi}$} (m-2-3)
(m-2-1) edge node [auto] {$\overline{\id \otimes \bgg}$} (m-1-2)
        edge node [auto, swap] {$\overline{\bff \otimes \id}$} (m-3-2)
(m-3-2) edge node [auto, swap] {$\overline{\psi}$} (m-2-3);
\end{tikzpicture}
\]
Let us describe the new objects added. By $\bff \otimes \bgg \colon A^n \otimes A^m \to A$ we denote the covector defined by $(\bff \otimes \bgg)(e \otimes e') := \bff(e) \bgg(e')$ for $e \in A^n$ and $e' \in A^m$. Notice that 
\[
\Img(\bff \otimes \bgg) = \langle f_ig_j \,/ \, i \in \{1,\dots, n\}, j \in \{1,\dots, m\}\rangle.
\] 
It is clear that the support of $\bff \otimes \bgg$ is contained in $Y$, in fact, it is the union of  the supports of $\bff$ and $\bgg$. The maps $\overline{\id \otimes \bgg}$ and $\overline{\bff \otimes \id}$ are also described similarly.

 \item[PS 2.] Let $\overline{\varphi}_1, \overline{\varphi}_2 \colon \bff \toto \bgg$ be two parallel arrows in $\SI_Y$. The axiom asks for an arrow that makes them equal. Let $N = \ker(\varphi_1 - \varphi_2)$. As $N$ is finitely generated, we may choose a surjection $A^p \epi N$ and let $\psi$ be the composition $A^p \epi N \to A^n$. Let $\bhh := \bff \circ \psi$. We have thus defined a morphism $\overline{\psi} \colon \bhh \to \bff$. It is clear that $\varphi_1 \psi = \varphi_2 \psi$. So, to conclude, we need to check that $\bhh \in \SI_Y$. But this is clear because the support of $\bhh$ agrees with the support of $\bff$.
\end{itemize}
 
To prove that it is cofiltered, observe that the category is clearly non empty, so it remains to prove that it is connected. Let $\bff$ and $\bgg$ be objects of $\SI_Y$. Let $\bff \oplus \bgg$ be the covector associated to the sequence $f_1, \dots, f_n, g_1, \dots, g_m$. The following diagram in $\SI_Y$ will suffice
 \[
\begin{tikzpicture} [baseline=(current  bounding  box.center)]
\matrix (m) [matrix of math nodes,
row sep=1em, column sep=2.5em,
text height=1.5ex, text depth=0.25ex]{
   & \bff &  \\ 
   &      & \bff \oplus \bgg \\
   & \bgg &  \\          
};
\path[->,font=\scriptsize,>=angle 90]
(m-1-2) edge  (m-2-3)
(m-3-2) edge  (m-2-3);
\end{tikzpicture}
\]
if we prove that $\bff \oplus \bgg \in \SI_Y$. The support of $\bff \oplus \bgg$ is the intersection of the supports of $\bff$ and $\bgg$, therefore it is contained in $Y$.
This completes the proof.
\end{proof}

\begin{cosa} \textbf{A system of Koszul complexes associated to $Y$.}
We denote by $K(\bff)$ the Koszul complex associated to $\bff$. See \cite[\S 17.4]{Eca}. As a graded module this is just $\bigwedge^{\bullet}\!A^n$, the exterior algebra of $A^n$ \emph{with the opposite grading}. Its differential is described shortly as contraction with the covector $\bff$. For an explicit description of the differential see \emph{loc. cit.} where it is denoted $\delta_{\bff}$. Notice that these complexes are perfect, as they are bounded and their objects are finitely generated free modules.

Let $\overline{\varphi} \colon \bff \to \bgg$ be  an arrow  in $\SI_Y$. It induces a map $K(\overline{\varphi}) \colon K(\bff) \to K(\bgg)$ which is just $\bigwedge^{\bullet}\!\varphi$. It is straightforward to check that it commutes with the differentials. This construction defines a functor $\SK_Y \colon \SI_Y \to \CCC(A\md)$.
\end{cosa}

\begin{cosa} \textbf{Algebraic supports on $Y$.}
 We define, using the category $\SI_Y$, a functor that assigns to a complex $M$ in the derived category, roughly speaking, the part supported in $Y$. Consider first the functor 
\[
\SK'_Y(M) \colon \SI_Y^{\op} \lto \CCC(A\md)
\]
defined by
\[
\SK'_Y(M)(\bff) = \Hom^{\bullet}(K(\bff), M).
\]
Notice that it yields a functor:
\[
\SK'_Y \colon \CCC(A\md) \lto \CCC\left(A\md^{\,\SI_Y^{\op}}\right)
\]
whose derived functor
\[
\R{}\SK'_Y \colon \D(A\md) \lto \D\left(A\md^{\,\SI_Y^{\op}}\right)
\]
satisfies that $\R{}\SK'_Y(M)$ corresponds to the diagram $\bff \mapsto \R{}\Hom^{\bullet}(K(\bff), M)$. The functor of complex of algebraic supports on $Y$ is an endofunctor 
\[
\SGamma_Y \colon \D(A\md) \to \D(A\md) \quad \text{defined as} \quad 
\SGamma_Y(M) = \LL{}\shdirlim{}\, \R{}\SK'_Y(M).
\]

Let us unfold this definition. Notice that $K(\bff)$ are $q$-projective because they are bounded complexes of free modules, therefore 
\[
\R{}\Hom^{\bullet}(K(\bff), M) \cong \Hom^{\bullet}(K(\bff), M).
\] 
On the other hand, by Lemma \ref{cofil}, the category $\SI_Y^{\op}$ is filtered, therefore directed colimits with values in modules are exact (\cfr~\cite[III, Corollaire 2.10.]{sga41}). In conclusion, the object $\SGamma_Y(M)$ may be represented in the derived category as
\begin{equation} \label{cartors}
 \SGamma_Y(M) = \dirlim{\bff \in \SI_Y}\!\Hom^{\bullet}(K(\bff), M).
\end{equation}
\end{cosa}

\begin{cosa} \textbf{The canonical transformation}.
With the previous notations, take for $\bff \in \SI_Y$ the natural morphism 
\begin{equation}\label{pseudoepi}
p_{\bff} \colon A \to K(\bff)
\end{equation}
which is the identity on degree 0. It induces a natural transformation
\begin{equation}\label{zetanat}
 \zeta_Y \colon \SGamma_Y(M) \lto M
\end{equation}
applying the contravariant functor $\shdirlim{}\Hom^{\bullet}(-, M)$ to the morphism \eqref{pseudoepi}.
\end{cosa}

\begin{prop}\label{gamistor}
 The pair $(\SGamma_Y, \zeta_Y)$ is an acyclization functor.
\end{prop}

\begin{proof}
 In view of \cite[Proposition 1.7]{AJS} we have to check that $\SGamma_Y$ is idempotent via $\zeta_Y$ \eqref{zetanat} and the two possible realizations of idempotence agree.
 To check idempotence, let $M \in \D (A\md)$. Consider the chain of isomorphisms:
 
\begin{align}
 \SGamma_Y\SGamma_Y(M) & = \dirlim{\bff \in \SI_Y}\Hom^{\bullet}(K(\bff), \dirlim{\bgg \in \SI_Y}\Hom^{\bullet}(K(\bgg), M)) \label{zeroth} \\
            & \cong \dirlim{\bff, \bgg \in \SI_Y}\Hom^{\bullet}(K(\bff) \otimes K(\bgg), M)
                                                                          \label{prim}\\
            & \cong \dirlim{\bff, \bgg \in \SI_Y}\Hom^{\bullet}(K(\bff \oplus \bgg), M) \label{sec}\\
            & \cong \dirlim{\bhh \in \SI_Y}\Hom^{\bullet}(K(\bhh), M) \label{terc}\\
            & = \SGamma_Y(M).  \notag  
\end{align}
Notice that the isomorphism \eqref{zeroth} is obtained applying twice \eqref{cartors} using finiteness of the Koszul complex, \eqref{prim} by adjunction and exactness of cofiltered colimits and \eqref{sec} because the exterior algebra transforms coproducts into tensor products. The last isomorphism, \eqref{terc} is just a relabelling of the objects in $\SI_Y$.

To end the proof we have to check that $\SGamma_Y(\zeta_M) = \zeta_{\SGamma_Y(M)}$. This follows for the commutativity of the following square
\[
\begin{tikzpicture}
\matrix (m) [matrix of math nodes,
row sep=3em, column sep=3em,
text height=2ex,
text depth=0.25ex]{
\SGamma_Y\SGamma_Y(M) & \SGamma_Y(M)  \\
\SGamma_Y(M)          & M  \\
};
\path[->, font=\scriptsize]
(m-1-1) edge node[above] {$\SGamma_Y(\zeta_M)$} (m-1-2)
        edge node[left] {$\zeta_{\SGamma_Y(M)}$} (m-2-1)
(m-1-2) edge node[right] {$\zeta_M$} (m-2-2)
(m-2-1) edge node[below] {$\zeta_M$} (m-2-2);
\end{tikzpicture}  
\]
Unravelling the definitions through \eqref{cartors} and applying \eqref{sec}, the square becomes
\[
\begin{tikzpicture}
\matrix (m) [matrix of math nodes,
row sep=3em, column sep=4em,
text height=2ex,
text depth=0.25ex]{
\shdirlim{}\Hom^{\bullet}(K(\bff \oplus \bgg), M) & \shdirlim{}\Hom^{\bullet}(K(\bff), M)  \\
\shdirlim{}\Hom^{\bullet}(K(\bgg), M)             & M  \\
};
\path[->, font=\scriptsize]
(m-1-1) edge node[above] {via $\bff \to \bff \oplus \bgg $} (m-1-2)
        edge node[left] {via $\bgg \to \bff \oplus \bgg$} (m-2-1)
(m-1-2) edge node[right] {via $p_{\bff}$} (m-2-2)
(m-2-1) edge node[below] {via $p_{\bgg}$} (m-2-2);
\end{tikzpicture}  
\]
that is clearly commutative.
\end{proof}

\begin{cosa} \textbf{The localization triangle associated to $\SGamma_Y$}.
Let $p_{\bff} \colon A \to K(\bff)$ be the morphism defined in \eqref{pseudoepi}. It is clear that it has a graded section so it is inserted in a so-called semi-split exact sequence of complexes
\begin{equation}\label{kaidce}
 0 \lto A \lto K(\bff) \lto C(\bff)[1] \lto 0.
\end{equation}

Recall the notation of brutal truncations, \ie for $M \in \CCC(A\md)$, define $\sigma^{\geq n} M$ as the subcomplex of the complex $M$ such that $(\sigma^{\geq n} M)^i = 0$  if $i<n$ and $(\sigma^{\geq n} M)^i = M^i$ if $i \geq n$. Also define $\sigma^{\leq n+1} M := M/\sigma^{\geq n} M$.

With this notation, $C(\bff)= \sigma^{\leq -1}K(\bff)[-1]$. The colimit in $n$ of the complexes  $C(\bff^n)$ taking appropriate powers of $\bff$ is identified with the \v{C}ech complex on the complementary open subset to the support of $\langle \bff \rangle$, but we will not make use this fact. 

From the sequence \eqref{kaidce}, by translation, we obtain a distinguished triangle in $\D(A\md)$
\[
C(\bff) \lto A \lto K(\bff) \overset{+}\lto .
\]
Applying the contravariant functor $\shdirlim{}\Hom^{\bullet}(-, M)$ as before we obtain the distinguished triangle
\[
\SGamma_Y(M) \lto M \lto \SL_Y(M) \overset{+}\lto 
\]
where we define
\begin{equation}\label{carloc}
 \SL_Y(M) = \dirlim{\bff \in \SI_Y}\!\Hom^{\bullet}(C(\bff), M)
\end{equation}
similarly to \eqref{cartors}. If follows from Proposition \ref{gamistor} that $\SL_Y$ is a localization functor in view of \cite[Proposition 1.6]{AJS}. 
\end{cosa}

\begin{thm} \label{mainloc}
 Let $A$ be a noetherian ring, $Y \subset \spec(A)$ a specialization closed subset and $\ell_Y$ the corresponding localization given by \cite{Nct}, then $\SL_Y = \ell_Y$.
\end{thm}

\begin{proof}
Proving that $\SL_Y  = \ell_Y$ is equivalent to see that $\SGamma_Y = \gamma_Y$. As remarked before, it is enough to check that $\SGamma_Y$ satisfies \eqref{caracyc}, \ie~ how $\SGamma_Y$ acts on residue fields.
 
 Let $\ip \in \spec(A)$. Let us look first at the case $\ip \notin Y$. This is equivalent to the fact that $\langle \bff \rangle \not\subset \ip$ for every $\bff \in \SI_Y$. We have to show that $\SGamma_Y( \kappa(\ip)) = 0$. Let $K^{\vee}(\bff) := \Hom^{\bullet}(K(\bff), A)$. The Koszul complex is auto-dual up to shift by \cite[Proposition 17.15.]{Eca}, namely, we have that $K^{\vee}(\bff) \cong K(\bff)[-n]$. We will use this fact without further notice. It follows that to see that
 \[
 \dirlim{\bff \in \SI_Y}\Hom^{\bullet}(K(\bff),  \kappa(\ip)) = 0
 \]
it suffices to show that $\Hom^{\bullet}(K(\bff),  \kappa(\ip)) =  K^{\vee}(\bff) \otimes \kappa(\ip)= 0$, for every $\bff$ such that $\Img(\bff) \not\subset \ip$. But we may represent $\bff = (f_1, \dots, f_n)$, as $\bff = f_1 \oplus \dots \oplus f_n$, then $K(\bff) = K(f_1) \otimes \dots \otimes K(f_n)$. Also $K^{\vee}(f_i) = K(f_i)[-1]$ for all $i \in \{1, \dots, n\}$. With all this in mind, the fact that $\langle \bff \rangle \not\subset \ip$ implies that there is an $i \in \{1, \dots, n\}$ such that $f_i \notin \ip$. 
 
We are now bound to prove that $K(f_i) \otimes \kappa(\ip) = 0$. By \cite[Lemma 17.13.]{Eca} the self map $\mu_{f_i} \colon K(f_i) \to K(f_i)$ defined by multiplication by $f_i$ is homotopic to 0, and therefore so is $\mu_{f_i}\otimes \id  \colon K(f_i) \otimes \kappa(\ip) \to K(f_i) \otimes \kappa(\ip)$. But this map is identified with $\mu_{f_i} \colon (K(f_i) \otimes A/\ip))_{\ip} \to (K(f_i) \otimes A/\ip))_{\ip}$ that is an isomorphism with inverse $\mu_{1/f_i}$ that exists because $f_i \notin \ip$. Thus the conclusion follows.

Let us now treat the case $\ip \in Y$, in other words, there is an $\bhh \in \SI_Y$ with $\bhh = (h_1, \dots, h_m)$ such that $\Img(\bhh) \subset \ip$. We have to show that $\SGamma_Y(\kappa(\ip)) = \kappa(\ip)$. We have the following chain of isomorphisms
\begin{align*}
\dirlim{\bff \in \SI_Y}\Hom^{\bullet}(K(\bff), \kappa(\ip)) 
           & \cong \dirlim{\bff \in \SI_Y} (K^{\vee}(\bff) \otimes \kappa(\ip))
                     \tag{$K(\bff)$ is perfect}\\
           & \cong \dirlim{\bff \in \SI_Y} (K^{\vee}(\bff) \otimes A/\ip)_{\ip} \\
           & \cong \dirlim{\bff \in \SI_Y} (K(\bff) \otimes A/\ip)_{\ip}[-n].
                     \tag{selfduality of $K(\bff)$}
\end{align*}
Let us compute the homologies of this last complex. By exactness of filtered direct limits, it is enough to compute the limits at each degree.

To compute at the zero level, we notice that $ H^0(K(\bff)) = A/\langle \bff \rangle$ for $\bff \in \SI_Y$, so $H^0(K(\bff) \otimes  \kappa(\ip)) = \kappa(\ip)$, therefore 
\[
\h^0\big(\!\!\dirlim{\bff \in \SI_Y}\Hom^{\bullet}(K(\bff), \kappa(\ip))\big) = \kappa(\ip).
\]
Now let $i \in \NN$ with $i \in \{1, \dots, n\}$. Then $K^{-i}(\bff) = \bigwedge^{i}A^n$. Let $\varphi_{\bhh} \colon A^n \otimes A^m \to A^n$ defined by $\varphi_{\bhh}(e_i \otimes e_j) = h_j e_i$, where $m$ denotes the length of the sequence $\bhh$. It is clear that for every $\bff \in \SI_Y$, it holds that $\bff \circ \varphi_{\bhh} = \bff \otimes \bhh$. The corresponding map in $\SI_Y$ is denoted $\overline{\varphi_{\bhh}} \colon \bff \otimes \bhh \to  \bff$. It induces a map 
\[
K(\overline{\varphi}_{\bhh}) \colon K(\bff \otimes \bhh ) \to K(\bff)
\]
that, after tensoring by $A/\ip$, at each level (identifying $A^n \otimes A^m$ with $A^{n m}$) is
\[
\bigwedge^{i}\varphi_{\bhh} \otimes \id \colon \bigwedge^{i}A^{n m} \otimes A/\ip 
                                  \lto \bigwedge^{i}A^n \otimes A/\ip
\]
with $i > 0$. But $\Img(\bhh) \subset \ip$ implies that the transition maps of the diagram $\bigwedge^{i}\varphi_{\bhh} \otimes \id$ whose entries are multiples of the components of $\bhh$, act as zero. In conclusion
\[
\h^i\big(\!\!\dirlim{\bff \in \SI_Y}\Hom^{\bullet}(K(\bff), \kappa(\ip))\big) = 0
\]
for every $i \neq 0$, and this finishes the proof.
\end{proof}

\begin{cosa}\textbf{Smashing localization.}\label{smash}
 According to \cite[Corollary 3.4.]{Ntc} among the localizations in $\D(A\md)$, the smashing localizations correspond to specialization closed subsets $Y \subset \spec(A)$. We remind the reader that a localization is smashing if the localization functor commutes with coproducts. With our description, it is clear that $\SL_Y = \ell_Y$ commutes with coproducts. A feature of this description is that only compact objects in $\D(A\md)$ (\ie~perfect complexes) are used to compute localizations and acyclizations.
\end{cosa}

\begin{cosa}\textbf{An alternative approach.}
 A different method from \ref{defdelsys} in order to describe $\ell_Y$ can be given in terms of the ideals defining the closed subsets whose union equals the specialization closed subset $Y$. Let $\SJ_Y$ be the set of ideals of $A$ whose support is contained in $Y$. This is a poset so it can be viewed as a small category. It is also easy to see that it is cofiltered. Let the functors $\SGamma'_Y$ and $\SL'_Y$ be defined by the following
\begin{align}
 \SGamma'_Y(M) &= \dirlim{\ia \in \SJ_Y}\R\Hom^{\bullet}(A/\ia, M) \label{cartorsbis}\\
 \SL'_Y(M)     &= \dirlim{\ia \in \SJ_Y}\R\Hom^{\bullet}(\ia, M). \label{carlocbis}
\end{align}
analogues to \eqref{cartors} and \eqref{carloc}, respectively. As before, there is no need to derive the colimits because the systems are filtered. To handle the homs one may replace $A/\ia$ and $\ia$ by $q$-projective resolutions. Alternatively, one might use a $q$-injective resolution for $M$.

By arguing as in Proposition \ref{gamistor} one can check that $\SGamma'_Y$ is an acyclization functor. Also, using arguments close to the proof of Theorem \ref{mainloc}, one can show that $\SL'_Y = \ell_Y$. The index category is simpler but the objects involved, $A/\ia $, are not perfect complexes in general. That is why we choose the present construction, in line with the considerations in \ref{smash}.

The underived version of this formula goes back to Gabriel. Formula \eqref{carlocbis} goes sometimes under the name ``Deligne's formula'' because this author uses it in a fundamental way in the appendix of \cite{RD}.
\end{cosa}

\section{Derivators and group cohomology}\label{seis}

 In this last paragraph we are going to retrieve some results on group cohomology using the language of derivators. We begin by recalling some basic definitions. Fix an underlying commutative ring $A$. In the classical case $A = \ZZ$, but here everything works the same with this generality. As we are interested in $A$-linear representations, our framework will be the derivator $\PDD_{\!\!A\md}$.

\begin{cosa} \textbf{$G$-modules and the classifying category of a group}.
 Let $G$ be a group, an $A$-linear representation of $G$ is an $A$-module $M$ together with a group homomorphism
\begin{equation}\label{act}
 \lambda \colon G \lto \aut_A(M)
\end{equation}
 Notice that this definition is equivalent to the more familiar point of view of an action
 \[
 G \times M \lto M
 \]
 imposing linearity conditions on the action. 
 
 A group $G$ has associated a category $\SB G$ with only one element $\star$ and such that $\Hom_{\SB G}(\star,\star) = G$. For an $A$-linear representation of $G$, $M$, we will denote 
 \[
 \widetilde{M} \colon \SB G \lto A\md
 \]
 the functor defined by the representation $(M, \lambda)$. From now on we will refer to a representation $(M, \lambda)$ by its underlying module $M$ and leave the homomorphisms implicit.
\end{cosa}

\begin{cosa}\textbf{Invariants and coinvariants}.
 Let $M$ be an $A$-linear representation of $G$. We define the $A$-module of invariants as
 \[
 M^G = \{m \in M \, / \, gm = m, \, \forall g \in G\}
 \]
and the $A$-module of coinvariants as
\[
M_G = \scalebox{1.2}{$\frac{M}{\langle gm - m  \, / \, g \in G, \, m \in M \rangle}$}\,.
\]
It is a nice exercise to check that
\begin{equation} \label{coinvcolim}
 M^G = \invlim{\SB G}\widetilde{M} \qquad M_G = \dirlim{\SB G}\widetilde{M}
\end{equation}
\end{cosa}

\begin{cosa} \textbf{Complexes of representations}.
 A complex of representations may be viewed as a functor
 \[
 \widetilde{M} \colon \SB G \lto \CCC(A\md)
 \]
 which corresponds to a group homomorphism
\begin{equation*}
 G \lto \aut_{\CCC(A\md)}(M)
\end{equation*}
with $M \in \CCC(A\md)$. Notice that we may interpret $\widetilde{M} \in \CCC(A\md^{\SB G})$. In other words, an $A$-linear representation of $G$ (a module or, more generally, a complex) is just an element of $\PDC_{\!\!A\md}(\SB G)$.
\end{cosa}

\begin{cosa} \textbf{Derived representations}.
We define an $A$-linear \emph{derived representation} of $G$ as an element of $\PDD_{\!\!A\md}(\SB G)$. Every derived representation may be identified with a complex of representations.
\end{cosa}

\begin{cosa} \label{abssu} \textbf{The absolute setup}.
Let $c$ be the canonical functor $c \colon \SB G \to \es$. We recall from \ref{funnynot2} the chain of $\Delta$-functorial adjunctions
\[
\begin{tikzpicture}
      \node (G) {$\PDD_{\!\!A\md}(\SB G)$};  
      \node[node distance=3cm, right of = G] (H) {$\PDD_{\!\!A\md}(\es)$};
      \draw[->, bend left=30] (G) to node [above] {\footnotesize $\LL\shdirlim{}$} (H);
      \draw[->, bend right=30] (G) to node [below] {\footnotesize $\R\shinvlim{}$} (H);
      \draw[<-,] (G) to node [above] {\footnotesize $c^*$} (H);
\end{tikzpicture}
\]
where as before we denote $\LL\shdirlim{} = \LL{}c_!$ and $\R\shinvlim{} = \R{}c_*$.
\end{cosa}

\begin{prop} \label{homlim}
Let $M$ be an $A$-linear representation of $G$. We have canonical isomorphisms
\begin{align}
 \h_n(G,M) &\cong \h^{-n} \LL\!\!\dirlim{\SB G}\widetilde{M} \label{hog}\\
 \h^n(G,M) &\cong \h^{n} \R\!\!\invlim{\SB G}\widetilde{M} \label{cohog}
\end{align}
 for every $n \in \ZZ$.
\end{prop}

\begin{proof}
 Let us prove \eqref{hog}. The functor $\h_n(G,M)$ is the derived functor of the functor of coinvariants $M \rightsquigarrow M_G$, therefore both functors agree at level zero in view of \eqref{coinvcolim} so their derived functors are isomorphic. Similarly for \eqref{cohog} because in this case $\h^n(G,M)$ is the derived functor of the functor of invariants $M \rightsquigarrow M^G$.
\end{proof}

\begin{cosa} \textbf{Induced representations}.
Let $\phi \colon H \to G$ be any group homomorphism. Classical Frobenius reciprocity, extended to the complex context, asserts the existence of adjoints
\[
\begin{tikzpicture}
      \node (G) {$\CCC(A\md^{\SB G})$};  
      \node[node distance=4cm, right of = G] (H) {$\CCC(A\md^{\SB H})$};
      \draw[<-, bend left=30] (G) to node [above] {\footnotesize $\mathrm{Ind}^G_H$} (H);
      \draw[<-, bend right=30] (G) to node [below] {\footnotesize $\mathrm{CoInd}^G_H$} (H);
      \draw[->] (G) to node [above] {\footnotesize $\phi^\#$} (H);
\end{tikzpicture}
\]
where $\phi^\#$ is the usual restriction (or forgetful) functor (as in \cite[Definition 6.7.1]{W}) and $\mathrm{Ind}^G_H$ and $\textrm{CoInd}^G_H$ are usually called induction and coinduction functors, respectively.

Associated to $\phi \colon H \to G$, we have the functor $\SB\phi \colon \SB H \to \SB G$. With this notation, $\phi^\#$ is identified (at the level of complexes) with
\[
\SB\phi^* \colon \CCC(A\md^{\SB G}) \lto \CCC(A\md^{\SB H}).
\]
\end{cosa}

\begin{prop} \label{derfrobrep}
 (Derived Frobenius reciprocity). Let $\phi \colon H \to G$ be a group homomorphism. In the derivator $\PDD_{\!\!A\md}$ there is a chain of adjunctions
 \[
\begin{tikzpicture}
      \node (G) {$\PDD_{\!\!A\md}(\SB G)$};  
      \node[node distance=3cm, right of = G] (H) {$\PDD_{\!\!A\md}(\SB H)$};
      \draw[<-, bend left=30] (G) to node [above] {\footnotesize $\LL\SB\phi_!$} (H);
      \draw[<-, bend right=30] (G) to node [below] {\footnotesize $\R\SB\phi_*$} (H);
      \draw[->] (G) to node [above] {\footnotesize $\SB\phi^*$} (H);
\end{tikzpicture}
\]
where $\SB\phi^* = \phi^\#$, $\LL\SB\phi_! = \LL\mathrm{Ind}^G_H$ and $\R\SB\phi_* = \R\mathrm{CoInd}^G_H$.
\end{prop}

\begin{proof}
 Notice that $\SB\phi^*$ is an exact functor, therefore it induces a functor between derived categories. The rest is a consequence of Der3 (Proposition \ref{der3}) and the uniqueness of adjoints.
\end{proof}

\begin{cosa} \textbf{Derived Shapiro's Lemma}.
 We set the following notation. Let $H$ be a subgroup of a group $G$. Denote by $\iota \colon H \inc G$ the inclusion homomorphism. Also, let $c_G \colon \SB G \to \es$ and $c_H \colon \SB H \to \es$ be the canonical functors. 
\end{cosa}

\begin{prop} \label{DerSha}
 Let $\widetilde{M} \in \PDD_{\!\!A\md}(\SB H)$ be a derived representation. We have natural isomorphisms
\begin{align}
 \LL\!\!\dirlim{\SB G} \SB\iota_!\widetilde{M}  &\cong \LL\!\!\dirlim{\SB H}\widetilde{M} \label{HomSha}\\
 \R \!\!\invlim{\SB G} \SB\iota_*\widetilde{M}  &\cong \R\!\!\invlim{\SB H}\widetilde{M}  \label{CohSha}
\end{align}
\end{prop}

\begin{proof}
Notice that $(-)^*$ makes part of a 2-functor, therefore the adjoint constructions $\LL(-)_!$ and $\R(-)_*$ are 2-functors, so we have
\begin{align*}
\LL\!\!\dirlim{\SB G} \LL\SB\iota_!\widetilde{M} &\cong \LL\!\!\dirlim{\SB H}\widetilde{M}\\
\R \!\!\invlim{\SB G} \R\SB\iota_*\widetilde{M}  &\cong \R\!\!\invlim{\SB H}\widetilde{M}  
\end{align*}
Moreover, $\SB\iota_!$ and $\SB\iota_*$ are exact functors because $\SB \iota \colon \SB H \to \SB G$ is both a sieve and a cosieve and arguing as in \cite[Corollary 3.8]{gr13} we see that both $\SB\iota_!$  and  $\SB\iota_*$ have adjoints on both sides, therefore they are exact functors and there is no need to derive them.
\end{proof}

\begin{cor}
Let $M$ be an $A$-linear representation of $H$. There are canonical isomorphisms for every $n \in \ZZ$
\begin{align*}
 \h_n(G,\mathrm{Ind}^G_H(M))    &\cong \h_n(H,M) \\
 \h^n(G,\mathrm{CoInd}^G_H(M))  &\cong \h^n(H,M).
\end{align*}
\end{cor}

\begin{proof}
Consider the identifications $\SB\iota_! = \mathrm{Ind}^G_H$ and $\SB\iota_* = \mathrm{CoInd}^G_H$. Using these, in view of Proposition \ref{homlim}, combine Proposition \ref{derfrobrep} with \eqref{HomSha} for the first isomorphism, and with \eqref{CohSha} for the second.
\end{proof}

\begin{rem}
 Compare with \cite[\S 6.3]{W}.
\end{rem}

\begin{cosa} \textbf{The case of a normal subgroup.}
 We will see how to derive the classic Lyndon-Hochschild-Serre spectral sequence in our setup. For the rest of the section, we fix the following notation. Let $H$ be a normal subgroup of a group $G$ with $\iota \colon H \inc G$ the inclusion monomorphism. Let $m \colon G \to G/H$ be the canonical epimorphism. 
\end{cosa}

\begin{prop} \label{derLHS}
  (Derived Lyndon-Hochschild-Serre). In the previous setting, for a derived representation $\widetilde{M} \in \PDD_{\!\!\SA }(\SB G)$ there are isomorphisms of functors
\begin{align}
 \LL\!\!\dirlim{\SB \scalebox{0.7}{$\frac{G}{H}$}} \LL{}\SB m_!\widetilde{M}  
                      &\cong \LL\!\!\dirlim{\SB G}\widetilde{M} \label{HomLHS}\\
  \R\!\!\invlim{\SB \scalebox{0.7}{$\frac{G}{H}$}} \R{}\SB m_*\widetilde{M}  
                      &\cong \R\!\!\invlim{\SB G}\widetilde{M} \label{CohLHS}
\end{align}
\end{prop}

\begin{proof}
As in \ref{abssu}, denote by $c_G \colon \SB G \to \es$  and $c_{\scalebox{0.7}{$\frac{G}{H}$}} \colon \SB \scalebox{0.9}{$\frac{G}{H}$} \to \es$ the canonical functors. Let us consider also the functor $\SB m \colon \SB G \to \SB \scalebox{0.9}{$\frac{G}{H}$}$. It is clear that $c_{\scalebox{0.7}{$\frac{G}{H}$}} \circ \SB m = c_G$. As before, the fact that the $(-)^*$ makes part of a 2-functor, forces the adjoint constructions $\LL(-)_!$ and $\R(-)_*$ to be 2-functors, therefore we have
\begin{align*}
 \LL c_{{\scalebox{0.7}{$\frac{G}{H}$}} !} \LL{}\SB m_! &\cong \LL c_{G !} \\
  \R c_{{\scalebox{0.7}{$\frac{G}{H}$}} *} \R{}\SB m_*  &\cong \R c_{G *}
\end{align*}
and the results follow by applying these natural isomorphisms to $\widetilde{M}$. 
\end{proof}

For the derivation of the classic version of Lyndon-Hochschild-Serre theorem, we will decompose the argument into several steps. We will start with a couple of Lemmas.

\begin{lem}\label{preLHS}
Let $\SB G/\star_{G/H}$ and $\star_{G/H}\!\backslash \SB G$ be the slice and coslice categories, respectively, relative to the functor $\SB m \colon \SB G \to \SB \scalebox{0.9}{$\frac{G}{H}$}$. Let $\pi \colon \SB G/\star_{G/H} \to \SB G$ and $\varpi \colon \star_{G/H}\!\backslash \SB G \to \SB G$ denote the canonical functors as in \ref{expdesadj}. There are equivalences of categories
\begin{align*}
 \nu \colon  \SB H &\liso \SB G/\star_{G/H}\\
  \mu \colon  \SB H    &\liso \star_{G/H}\!\backslash \SB G
\end{align*}
 such that $\pi \circ \nu = \SB \iota$ and $\varpi \circ \mu = \SB \iota$, respectively.
\end{lem}
 
\begin{proof}
Let us start by describing the category $\SB G/\star_{G/H}$. Its objects are just the elements of $G$. The set of maps between $g, g' \in G$ are those $h \in H$ such that $hg' = g$ which is equivalent to  $g$ and $g'$ having the same class in $G/H$. In other words, $\Hom_{\SB G/\star_{G/H}}(g, g') = H = \aut_{\SB H}(\star)$. 

The functor $\nu$ takes $\star \in \SB H$ to the neutral element $1$ considered as an object in $\SB G/\star_{G/H}$. It follows that $\nu$ is fully faithful. But clearly all of the objects in $\SB G/\star_{G/H}$ are isomorphic, and this makes $\nu$ dense, whence an equivalence of categories as claimed. Finally the identification $\pi \circ \nu = \SB \iota$ is obvious at the level of objects and identifies the elements of $H$ as maps of $\SB H$ with themselves, viewed as elements of $G$, so as maps in $\SB G$.

For $\star_{G/H}\!\backslash \SB G$ everything works the same, the only difference is that the set of maps between $g, g' \in G$, considered as objects of this category, are those $h \in H$ such that $g' = gh$ and the rest of the arguments carry over, including the fact that $\varpi \circ \mu = \SB \iota$.
\end{proof}

\begin{lem}\label{compcoh}
 Let $M$ be an $A$-linear representation of $G$ and $\widetilde{M} \in \PDD_{\!\!\SA }(\SB G)$ its corresponding derived representation. We have the following isomorphisms
 \[
 \h_q(H, M) \cong \h^{-q} \LL{}\SB m_! \widetilde{M} \qquad
 \h^q(H, M) \cong \h^q \R{}\SB m_*     \widetilde{M}
 \]
 where on the left hand sides of each isomorphism we consider $M$ as an $A$-linear representation of $H$.
\end{lem}

\begin{proof}
Denote by $c_H \colon \SB H \to \es$ the canonical functor. By Proposition \ref{homlim} and by definition, respectively, we have that
\begin{align*}
 \h_q(H, M) &\cong \h^{-q} \LL{}c_{H !} \SB \iota^* \widetilde{M} 
             = \h^{-q} \LL\!\!\dirlim{\SB H} \SB \iota^* \widetilde{M} \\
 \h^q(H, M) &\cong \h^q \R{}c_{H !} \SB \iota^* \widetilde{M}
             = \h^q \R\!\!\invlim{\SB H} \SB \iota^* \widetilde{M} 
\end{align*}
On the other hand
\begin{align*}
 \LL{}\SB m_! \widetilde{M} 
          &= \LL{}\!\!\!\!\dirlim{\SB G/\star_{G/H}} (\widetilde{M} \circ \pi) 
                  \tag{by definition, see \ref{expdesadj}} \\
          &\cong \LL{}\!\!\dirlim{\SB H} (\widetilde{M} \circ \pi \circ \nu) 
                  \tag{Lemma \ref{preLHS}} \\
          &= \LL{}\!\!\dirlim{\SB H} (\widetilde{M} \circ \SB \iota) 
           = \LL{}\!\!\dirlim{\SB H} \SB \iota^* \widetilde{M}
\end{align*}
Analogously,
\begin{align*}
 \R{}\SB m_* \widetilde{M} 
          &= \R{}\!\!\!\!\invlim{\star_{G/H}\backslash\!\SB G} (\widetilde{M} \circ \varpi) 
                  \tag{by definition, see \ref{expdesadj}} \\
          &\cong \R{}\!\!\invlim{\SB H} (\widetilde{M} \circ \varpi \circ \mu) 
                  \tag{Lemma \ref{preLHS}} \\
          &= \R{}\!\!\invlim{\SB H} (\widetilde{M} \circ \SB \iota) 
           = \R{}\!\!\invlim{\SB H} \SB \iota^* \widetilde{M} 
\end{align*}
When invoking Lemma \ref{preLHS}, notice that an isomorphism of shape of diagrams does not alter neither the colimit nor the limit.
\end{proof}

\begin{thm}\label{LHSclas}
(Lyndon-Hochschild-Serre spectral sequence).
 Let $M$ be an $A$-linear representation of $G$. We have converging  spectral sequences
\begin{align*}
 E^2_{p q} = \h_p(G/H, \h_q(H, M))    &\imp \h_{p+q}(G,M) \\
 E_2^{p q} = \h^p(G/H, \h^q(H, M))    &\imp \h^{p+q}(G,M)  
\end{align*}
\end{thm}

\begin{proof}
By applying \cite[III, Theorem 7.7]{gman} (and its dual), we obtain at once the convergence of the spectral sequences
\begin{align*}
 E^2_{p q} = 
           \h^{-p}(\LL\!\!\dirlim{\SB \scalebox{0.7}{$\frac{G}{H}$}} \!\!
           \h^{-q}(\LL{}\SB m_! \widetilde{M}))  
           & \imp \h^{-(p+q)}(\LL\!\!\dirlim{\SB G}\widetilde{M}) \\
 E_2^{p q} = 
           \h^p(\R\!\!\invlim{\SB \scalebox{0.7}{$\frac{G}{H}$}} \!\!
           \h^q(\R{}\SB m_*\widetilde{M}))    
           & \imp \h^{p+q}(\R\!\!\invlim{\SB G}\widetilde{M})  
\end{align*}
We conclude by applying to the isomorphisms \eqref{HomLHS} and \eqref{CohLHS}, the computation in Lemma \ref{compcoh} combined with the isomorphisms \eqref{hog} and \eqref{cohog}.
\end{proof}

\begin{rem}
 The last theorem, which is a cornerstone in the theory of co/ho\-mo\-logy of groups follows simply from 2-functoriality. In the classical references, like \cite[6.8.2]{W}, the fact that a $G$-representation induces an $H$-representation goes implicit in the notation. This fact is made explicit in our approach by the appearance of the functor $\SB \iota^*$ and the identification of co/homology with respect to $H$ with the functors $\SB m_!$ and $\SB m_*$, respectively, via Lemma \ref{compcoh}.
\end{rem}

\begin{small}

\begin{ack}
 We thank the anonymous referees for their attentive reading and for suggesting several useful improvements.
\end{ack}

\begin{fund}
 This work has been partially supported by Spain's MINECO research project MTM2017-89830-P and Xunta de Galicia's ED431C 2019/10 both with E.U.'s FEDER funds. The second author is also partially supported by MCINN research project PID2020-115155GB-I00 with E.U.'s FEDER funds and by a MECD contract FPU 18/01203.
\end{fund}

%
%

\end{small}



\end{document}